\documentclass[sn-mathphys,Numbered]{sn-jnl}

\usepackage{graphicx}%
\usepackage{multirow}%
\usepackage{amsmath,amssymb,amsfonts}%
\usepackage{amsthm}%
\usepackage{mathrsfs}%
\usepackage[title]{appendix}%
\usepackage{xcolor}%
\usepackage{textcomp}%
\usepackage{manyfoot}%
\usepackage{booktabs}%
\usepackage{algorithm}%
\usepackage{algorithmicx}%
\usepackage{algpseudocode}%
\usepackage{listings}%
\usepackage{supertabular}%
\usepackage{caption}%
\usepackage{float}%
\usepackage{indentfirst}%
\usepackage{extarrows}%
\usepackage{subfigure}%
\usepackage{longtable}%
\usepackage{bm}%
\usepackage{hyperref}%
\usepackage{array}%
\usepackage{multirow}%
\newtheorem{assumption}{Assumption}%

\theoremstyle{thmstyleone}%
\newtheorem{theorem}{Theorem}
\newtheorem{proposition}[theorem]{Proposition}%

\theoremstyle{thmstyletwo}%

\theoremstyle{thmstylethree}%

\begin{document}

\title[Article Title]{Dual Newton Proximal Point Algorithm for Solution Paths of the $\ell_1$-Regularized Logistic Regression}


\author[1]{\fnm{Yong-Jin} \sur{Liu}}\email{yjliu@fzu.edu.cn}
\equalcont{These authors contributed equally to this work.}

\author*[2]{\fnm{Weimi} \sur{Zhou}}\email{wmzhou1997@163.com}
\equalcont{These authors contributed equally to this work.}

\affil[1]{\orgdiv{Center for Applied Mathematics of Fujian Province, School of Mathematics and Statistics}, \orgname{Fuzhou University}, \orgaddress{\street{No. 2 Wulongjiang North Avenue}, \city{Fuzhou}, \postcode{350108}, \state{Fujian}, \country{China}}}

\affil*[2]{\orgdiv{School of Mathematics and Statistics}, \orgname{Fuzhou University}, \orgaddress{\street{No. 2 Wulongjiang North Avenue}, \city{Fuzhou}, \postcode{350108}, \state{Fujian}, \country{China}}}


\abstract{The $\ell_1$-regularized logistic regression  is a widely used statistical model in data classification. This paper proposes a dual Newton method based proximal point algorithm (PPDNA) to solve the $\ell_1$-regularized logistic regression problem with bias term. The global and local convergence of PPDNA hold under mild conditions. The computational cost of a semismooth Newton ({\sc Ssn}) algoithm  for solving subproblems in the PPDNA can be effectively reduced by fully exploiting the second-order sparsity of the problem. We also design an adaptive sieving (AS) strategy to generate solution paths for the $\ell_1$-regularized logistic regression problem, where each subproblem in the AS strategy is solved by the PPDNA. This strategy exploits active set constraints to reduce the number of variables in the problem, thereby speeding up the PPDNA for solving a series of problems. Numerical experiments demonstrate the superior performance of the PPDNA in comparison with some state-of-the-art second-order algorithms and the efficiency of the AS strategy combined with the PPDNA for generating solution paths.}




\maketitle

\section{Introduction}\label{intro}
High-dimensional logistic regression model is a classification model commonly used in machine learning. In logistic regression, given pairs of training samples $({\bm{a}}_1,{\bm{b_1}}),\ldots,({\bm{a}}_m,{{{\bm{b}}_m}})\in\mathbb{R}^n\times\{-1,1\}$,  the conditional probability distribution of label ${\bm{b}}_i$ given  a vector ${\bm{a}}_i$ is defined by 
$$
p_{\log}(v,{\bm{w}})_i=p({\bm{b}}_i|{\bm{a_i}})=\frac{1}{1+\exp(-{\bm{b}}_i({\bm{a}}_i^{\top}{\bm{w}}+v))},\ \forall i=1,\ldots,m,
$$ 
where the weight vector ${\bm{w}}\in \mathbb{R}^n$ and the intercept $v\in\mathbb{R}$ are unknown parameters. The corresponding maximum log-likelihood function is as follows:
\begin{equation*}
	\begin{split}
		\max_{{\bm{w}}\in\mathbb{R}^n,v\in\mathbb{R}}\log(\prod_{i=1}^{m} p_{\log}(v,{\bm{w}})_i)&=\max_{{\bm{w}}\in\mathbb{R}^n,v\in\mathbb{R}}\sum_{i=1}^{m}\log p_{\log}(v,{\bm{w}})_i\\
		&=-\min_{{\bm{w}}\in\mathbb{R}^n,v\in\mathbb{R}}\sum_{i=1}^{m}\log({1+\exp(-{\bm{b}}_i({\bm{a}}_i^{\top}{\bm{w}}+v))}),
	\end{split}
\end{equation*}
where $ \sum_{i=1}^{m}\log({1+\exp(-{\bm{b}}_i({\bm{w}}^{\top}{\bm{a}}_i+v))})$ is called the logistic loss. 
The problem of minimizing the average logistic loss is called the logistic regression problem.  Logistic regression model is a supervised statistical learning method that can be used for classification, prediction, and so on. In classification, it can not only achieve binary classification, but also be extended to multi-class classification problems \cite{H.L.S.2013,C.X.G.2018}. In the context of the era of big data, it is challenging to solve logistic regression model when the number $n$ of features is much larger than the number $m$ of samples. In order to avoid overfitting, it is usually necessary to select some main features and exclude some irrelevant variables. The $\ell_1$ regularization is widely used in machine learning, which automatically makes the model filter features \cite{N.2004,L.L.A.2006}.  
In this paper, we consider the $\ell_1$-regularized logistic regression problem with a bias term as follows:
\begin{equation}\label{1.1}
	\mathop {\min }\limits_{{\bm{w}}\in \mathbb{R}^n,v\in\mathbb{R}} \frac{1}{m}\sum_{i=1}^{m}\log(1+\exp(-{\bm{b}}_i({\bm{a}}_i^{\top}{\bm{w}}+v)))+\lambda\|{\bm{w}}\|_1, 
\end{equation}
where $\lambda>0$ is a given regularization parameter. The logistic regression with $\ell_1$ regularization  has many applications in statistical learning such as high-dimensional gene selection \cite{BDK.2020},  cancer classification \cite{A.R.B.2021}, and  graphical model selection \cite{W.L.R.2006}.

The $\ell_1$-regularized logistic regression problem is a widely used machine learning model. However, it is difficult to solve it because the $\ell_1$ regularization is not differentiable.  To address this problem, we shall briefly review existing methods and recent sieving strategies for this problem. For the $\ell_1$-regularized  logistic regression problem without bias term, many efficient optimization methods have been proposed. Lee et al. \cite{L.L.A.2006} reformulated the quadratic approximation of the $\ell_1$-regularized  logistic regression problem as an $\ell_1$ constrained least squares problem by an iteratively reweighted least squares formulation and then applied least angle regression (LARS) to solve it at each iteration. Milzarek et al. \cite{M.X.C.W.U.2019} proposed a globalized stochastic semismooth Newton method for solving $\ell_1$-regularized  logistic regression problem. In addition, there are some popular methods for solving convex composite optimization problems via quadratic approximation schemes such as the proximal Quasi-Newton method \cite{G.SK.2018}, improved GLMNET method \cite{Y.H.L.2012}, inexact regularized proximal Newton method \cite{Y.Z.S.2019}, and proximal Newton-type method \cite{M.Y.Z.2022}. For the $\ell_1$-regularized logistic regression problem \eqref{1.1} with the bias term, Koh et al. \cite{K.K.B.2007} propose an interior point method which applys a PCG method to find the search direction to solve the large-scale sparse problem \eqref{1.1}.  Furthermore, considering the sparsity of the solution,  there are many screening strategies attempt to solve the lasso problem. Tibshirani et al.  \cite{T.R.J.J.2012} proposed a simple strong rule that screens out far more predictors than the safe screening rule \cite{G.V.R.2010} but this improvement comes at the cost of incorrectly discarding the possibility of active predictors. Wang et al. \cite{W.J.J.2013}  proposed an efficient and effective screening rule via dual polytope projections (DPP) for lasso.  In terms of algorithm design, it is impressive that the  dual semismooth Newton method based proximal point algorithm (PPDNA) \cite{L.L.S.T.2019,L.S.T.2018,LXD.S.T.2018,ZYJ.Z.S.2020} has good numerical performance on  solving large-scale convex composite optimization problems, including  the exclusive Lasso model \cite{L.S.T.2019}, the group graphical Lasso model \cite{Z.Z.S.2020},  and the Dantzig selector \cite{F.L.X.2021}.  
The adaptive sieving (AS) strategy is a feature screening rule proposed in \cite{L.Y.S.2020} for exclusive lasso regularization.  Compared with other screening rules \cite{T.R.J.J.2012,W.J.J.2013,G.V.R.2010}, the adaptive sieving strategy can be applied to a more general regularizer that  induce solution sparseness. The main idea of the adaptive sieving strategy is to reduce the number of variables in the problem to improve efficiency of solving a series of the large scale problems \eqref{1.1} with sparse structure \cite{L.J.S.2021,B.L.2022,L.S.D.2022}. 


In this paper,  inspired by \cite{L.Y.S.2020,L.S.T.2019,L.S.T.2018,K.K.B.2007,Y.C.S.2021},   we shall design an efficient algorithm based on the dual semismooth Newton method for directly solving problem \eqref{1.1} with a bias term and employ an adaptive sieving technique based on the Karush-Kuhn-Tucker (KKT) conditions to generate solution paths of problem \eqref{1.1}. The reason why we choose to solve problem \eqref{1.1} directly is that we later consider the extension of the algorithms application in some compound convex optimization problems with relatively complex sparse regularizers, such as fused lasso, cluster lasso and other regularizers. Such problems cannot be equivalently transformed into univariate problems. Solving problem \eqref{1.1} with two variables directly has reference significance for us to apply PPDNA and sieving strategy to such problems in the future. In addition, the current way to generate the solution path of  problem \eqref{1.1} is generally to use algorithms to solve a series of problems. However, generating the solution path of large-scale problem \eqref{1.1} will incur an expensive time cost.  In order to improve the efficiency of generating solution paths, we adopt an adaptive sieving strategy to generate solution paths for a series of $\lambda$. Different from \cite{L.Y.S.2020}, we apply the adaptive sieving  strategy to the convex composite optimization problem with $\ell_1$-regularization. The existing second-order algorithms solve the $\ell_1$-regularized  logistic regression problem without bias  term as a special example. Considering the good performance of the adaptive sieving strategy in solving convex optimization problems with sparse regularizer, we try to combine the adaptive sieving strategy with these existing second-order algorithms to see its power.

The remaining parts of this paper are organized as follows. Sect.~\ref{sec:2} is devoted to exploring the  PPDNA for solving problem \eqref{1.1}, in which a semismooth Newton algorithm is applied to solve its internal subproblems. In Sect.~\ref{sec:3}, we combine the PPDNA with the AS strategy for generating solution paths of problem \eqref{1.1}. In Sect.~\ref{sec:4}, we compare our algorithms with some second-order methods on real and random data.  We make the conclusion of this paper in Sect.~\ref{sec:5}.

{\bf{Notation and preliminaries:}} 
The following notations and terminologies are used throughout the paper. We use  $``;"$ for adjoining vectors in a column. For given positive integer $m$, we denote ${\bm{I}}_m$ and ${\bm{1}}_m$ as the $m\times m$ identity matrix and the column vector of all ones, respectively. For given ${\bm{x}}\in \mathbb{R}^n$, we use $``|{\bm{x}}|"$ to denote the absolute vector whose $i$-th entry is $|{\bm{x}}_i|$ and $``{\rm{sgn}}({\bm{x}})"$ to denote the sgn vector whose $i$-th entry is $1$ if  ${\bm{x}}_i>0,$ $-1$ if ${\bm{x}}_i<0$, and $0$ otherwise. Denote $``{\rm{Diag}}({\bm{x}})"$ as the diagonal matrix whose diagonal is given by vector ${\bm{x}}$. For any self-adjoint positive semidefinite linear operator $\mathcal{M}:\mathbb{R}^n\to\mathbb{R}^n,$ we define $\langle {\bm{x}},{\bm{x}}'\rangle_{\mathcal{M}}:=\langle {\bm{x}},{\mathcal{M}}{\bm{x}}'\rangle$ and $\|{\bm{x}}\|_{\mathcal{M}}:=\sqrt{\langle {\bm{x}},{\bm{x}}\rangle_{\mathcal{M}}}$ for all ${\bm{x}},{\bm{x}}'\in \mathbb{R}^n$. We denote  the largest and smallest eigenvalues of $\mathcal{M}$ by $\lambda_{\max}(\mathcal{M})$ and $\lambda_{\min}(\mathcal{M})$, respectively. For given subset $\mathcal{C}\subseteq\mathbb{R}^n$, we define the weighted distance of ${\bm{x}}\in\mathbb{R}^n$  to  $\mathcal{C}$ by ${\rm dist}_{\mathcal{M}}({\bm{x}},\mathcal{C}):= \inf_{{\bm{x}}'\in\mathcal{C}}\|{\bm{x}}-{\bm{x}}'\|_{\mathcal{M}}$.
The $\ell_{\infty}$-norm unit  ball is defined by $\mathcal{B}_{\infty}:=\{{\bm{x}}\in\mathbb{R}^n | \ \|{\bm{x}}\|_{\infty}\leq 1 \}$. We use $``\circ"$ to denote the Hadamard  product.  

For a closed proper convex function $h:\mathbb{R}^m\to(-\infty,+\infty]$, the conjugate function of  $h(\cdot)$ is defined by $h^*({\bm{y}}):=\sup_{{\bm{x}}\in\mathbb{R}^m}\{\langle {\bm{x}},{\bm{y}}\rangle-h({\bm{x}})\}$.
We define a closed proper convex function $h(\cdot)$ as 
\begin{equation}\label{def1}
	h({\bm{x}}): = \frac{1}{m}\sum_{i=1}^{m}\log(1+\exp(-{\bm{b}}_i{\bm{x}}_i)),\ \forall {\bm{x}}\in\mathbb{R}^m.
\end{equation}
The the gradient of $h(\cdot)$ at ${\bm{x}}\in\mathbb{R}^m$ is given by
\begin{equation}\label{grad}
	\nabla h(x)_i=\frac{-b_i\exp(-b_ix_i)}{m(1+\exp(-b_ix_i))},\ i=1,\dots,m.
\end{equation}
Then the conjugate function of $h(\cdot)$ is obtained by
\begin{equation}\label{def2}
	h^*({\bm{y}})=-\frac{1}{m}\sum_{i=1}^{m}\log(\frac{{\bm{b}}_i}{{\bm{b}}_i+m{\bm{y}}_i})+\sum_{i=1}^{m}\frac{{\bm{y}}_i}{{\bm{b}}_i}\log(-\frac{{\bm{b}}_i}{m{\bm{y}}_i}-1),\ \forall {\bm{y}}\in {\rm{dom}}h^*,
\end{equation} 
where ${\rm{dom}}h^*=\{{\bm{y}}\in\mathbb{R}^m| -(1/m){\bm{1}}_m< {\bm{y}}\circ {\bm{b}}<{\bm{0}} \}$. It can be found that $h^*(\cdot)$ is strongly convex  and twice continuously differentiable. In addition, the gradient and Hessian of $h^*(\cdot)$ at ${\bm{y}}\in {\rm{dom}}h^*$ are respectively given by
\begin{align}
	&(\nabla h^*({\bm{y}}))_i=(1/{\bm{b}}_i)\log(-{\bm{b}}_i/(m{\bm{y}}_i)-1),\ i=1,\ldots,m,\label{def3}\\
	&\nabla^2 h^*({\bm{y}})={\rm{Diag}}({\hat{\bm{y}}}),\ {\rm{where}}\ \hat{{\bm{y}}}_i=-1/(({\bm{b}}_i+m{\bm{y}}_i){\bm{y}}_i),\ i=1,\ldots,m.\label{def4}
\end{align} 
For a closed proper convex function $f:\mathbb{R}^n\to(-\infty,+\infty]$, the Moreau-Yosida regularization and the proximal mapping of $f$ at ${\bm{x}}$ is defined by
\begin{equation*}
	\begin{split}
		E_f({\bm{x}}) & :=\min\limits_{{\bm{y}}\in \mathbb{R}^n}\{f({\bm{y}})+\frac{1}{2}\|{\bm{y}}-{\bm{x}}\|^2\},\ \forall {\bm{x}}\in \mathbb{R}^n,\\ {\rm{Prox}}_f({\bm{x}}) & :=\mathop{\arg\min}\limits_{{\bm{y}}\in \mathbb{R}^n}\{f({\bm{y}})+\frac{1}{2}\|{\bm{y}}-{\bm{x}}\|^2\},\ \forall {\bm{x}}\in \mathbb{R}^n.
	\end{split}
\end{equation*}
It is well known that $E_f(\cdot)$ is convex, continuously differentiable, and its gradient at ${\bm{x}}\in\mathcal{X}$ is $\nabla E_f({\bm{x}})={\bm{x}}-{\rm{Prox}}_f({\bm{x}})$. Furthermore, $\nabla E_f(\cdot)$ and ${\rm{Prox}}_f(\cdot)$ are globally Lipschitz continuous with modulus $1$  \cite{LC.SC.1997}.
Specifically, for given ${\bm{s}}\in\mathbb{R}^n_{++}$,  the proximal mapping of the weighted $\ell_1$-norm is  given by 
$$ 
{\rm{Prox}}_{{\bm{s}}^{\top}|\cdot|}({\bm{x}})={\rm{sgn}}({\bm{x}})\circ\max\{|{\bm{x}}|-{\bm{s}},{\bm{0}}\},\ \forall {\bm{x}}\in \mathbb{R}^n.
$$ 
It is obvious that $\ell_1$-norm is a special case of the weighted $\ell_1$-norm, that is, ${\bm{s}}=t{\bm{1}}_n$ where $t>0$, then it holds
$$
{\rm{Prox}}_{t\|\cdot\|_1}({\bm{x}})={\rm{sgn}}({\bm{x}})\circ\max\{|{\bm{x}}|-t{\bm{1}}_n,{\bm{0}}\},\ \forall {\bm{x}}\in \mathbb{R}^n.
$$
	\section{A dual Newton method based proximal point algorithm} \label{sec:2}
In this section, we shall introduce the specific details of a dual Newton method based proximal point algorithm (PPDNA) \cite{L.S.T.2019} for solving the equivalent form of problem \eqref{1.1}. Then we establish the global and local convergence of the PPDNA.

Define ${\bm{A}}$ as ${\bm{A}}=[{\bm{a}}_1,...,{\bm{a}}_m]^{\top}\in\mathbb{R}^{m\times n}$,  we rewrite problem \eqref{1.1}  as
\begin{equation}\label{2.1}\tag{${\rm{P}}_{\lambda}$}
	\mathop {\min }\limits_{{\bm{w}}\in\mathbb{R}^n,v\in\mathbb{R}}\  \{f({\bm{w}},v):=h({\bm{A}}{\bm{w}}+v{\bm{1}}_m)+\lambda p({\bm{w}})\}, 
\end{equation}
where $h(\cdot)$ is the  loss function defined in \eqref{def1}, $p(\cdot)=\|\cdot\|_1$ is the
$\ell_1$-regularization function. It is clear that $h(\cdot)$ is a twice continuously differentiable convex function, and $p(\cdot)$ is a closed proper convex function. Denote the optimal solution set of problem \eqref{2.1} by $\Omega_{\lambda}$ and  the Karush-Kuhn-Tucker (KKT) residual function corresponding to problem \eqref{2.1} by
$$
R_{\lambda}({\bm{w}},v):= \left [                
\begin{array}{c}    
	{\bm{w}}-{\rm{Prox}}_{\lambda p}({\bm{w}}-{\bm{A}}^{\top}\nabla h({\bm{A}}{\bm{w}}+v{\bm{1}}_m)) \\
	\nabla h({\bm{A}}{\bm{w}}+v{\bf{1}}_m)^{\top}{\bm{1}}_m
\end{array}
\right ]. 
$$
The KKT conditions imply that $({\bm{w}}^*,v^*)\in \Omega_{\lambda}$ if and only if $R_{\lambda}({\bm{w}}^*,v^*)={\bm{0}}.$ In this paper, we assume that $\Omega_{\lambda}$ is nonempty and compact. Actually, as stated in \cite[Section 2.1]{Z.SAM.2017}, this assumption is reasonable.

Problem \eqref{2.1} can be equivalently written as 
\begin{equation}\label{2.2}
	\mathop {\min }\limits_{{\bm{w}}\in\mathbb{R}^n,v\in\mathbb{R},{\bm{y}}\in\mathbb{R}^m} \{h({\bm{y}})+\lambda p({\bm{w}})\ | \ {\bm{y}}={\bm{A}}{\bm{w}}+v{\bm{1}}_m\}.
\end{equation}
The KKT conditions for problem \eqref{2.2} are formulated as follows:
\begin{equation}\label{2.3}
	\nabla h({\bm{y}})-{\bm{u}}={\bm{0}},\ -{\bm{A}}^{\top}{\bm{u}}\in \partial\lambda p({\bm{w}}),\ {\bm{y}}-{\bm{A}}{\bm{w}}-v{\bm{1}}_m={\bm{0}},\ {\bm{u}}^{\top}{\bm{1}}_m=0,
\end{equation}
where ${\bm{u}}\in\mathbb{R}^m$ is the Lagrange multiplier. For any given $({\bm{w}},v,{\bm{y}},{\bm{u}})\in \mathbb{R}^n\times \mathbb{R}\times\mathbb{R}^m\times\mathbb{R}^m$,
the KKT residual function of problem \eqref{2.2} is defined by
\begin{equation} \nonumber  
	\hat{R}_{\lambda}({\bm{w}},v,{\bm{y}},{\bm{u}}):= \left [                
	\begin{array}{c}   
		\nabla h({\bm{y}})-{\bm{u}} \\  
		{\bm{w}}-{\rm{Prox}}_{\lambda p}({\bm{w}}-{\bm{A}}^{\top}{\bm{u}}) \\
		{\bm{y}}-{\bm{Aw}}-v{\bm{1}}_m \\
		{\bm{u}}^{\top}{\bm{1}}_m
	\end{array}
	\right ].         
\end{equation}
Since problem \eqref{2.1} and problem \eqref{2.2} are equivalent, we can know that their corresponding KKT conditions are also equivalent.

Given the initial points ${\bm{w}}^0\in\mathbb{R}^n,v^0\in\mathbb{R}$, the preconditioned proximal point algorithm \cite{L.S.T.2020} generates sequences $\{{\bm{w}}^k\}\subseteq\mathbb{R}^n,\{v^k\}\subseteq\mathbb{R}$ by the following proximal rule:
\begin{equation}\label{2.5}
	\begin{split}
		({\bm{w}}^{k+1};v^{k+1})\approx\mathcal{P}_k({\bm{w}}^k,v^k):=&\mathop {\arg\min }\limits_{{\bm{w}}\in\mathbb{R}^n,v\in\mathbb{R}} \{ h({\bm{Aw}}+v{\bm{1}}_m)+\lambda p({\bm{w}})\\
		&+\frac{1}{2\sigma_k}\|{\bm{w}}-{\bm{w}}^k\|^2+\frac{1}{2\gamma_k}(v-v^k)^2\},
	\end{split}
\end{equation}
where both $\{\sigma_k\}$ and $\{\gamma_k\}$ are sequences of positive real numbers. 
The corresponding Lagrangian function of problem \eqref{2.5} is 
\begin{equation*}\label{2.6}
	\begin{split}
		L({\bm{w}},v,{\bm{y}};{\bm{u}}) = & h({\bm{y}})+\lambda p({\bm{w}})
		+\frac{1}{2\sigma_k}\|{\bm{w}}-{\bm{w}}^k\|^2+\frac{1}{2\gamma_k}(v-v^k)^2\\
		&+\left\langle {\bm{u}},{\bm{Aw}}+v{\bm{1}}_m-{\bm{y}}\right\rangle. 
	\end{split}
\end{equation*}
The dual of problem \eqref{2.5} admits the following minimization form:
\begin{equation*}\label{2.7}
	\min\limits_{{\bm{u}}\in {\rm{dom}}h^*} \left(\psi_k({\bm{u}})=-\inf_{\bm{w},v,\bm{y}} 	L({\bm{w}},v,{\bm{y}};{\bm{u}}) \right)
\end{equation*}
where  ${\rm{dom}}h^*$ is described in \eqref{def2} and the dual function $\psi_k(\cdot)$ is given by 
\begin{equation}\label{2.8}
	\begin{split}
		\psi_k({\bm{u}})=&h^*({\bm{u}})-\frac{1}{\sigma_k}E_{\sigma_k\lambda p}({\bm{w}}^k-\sigma_k{\bm{A}}^{\top}{\bm{u}})+\frac{1}{2\sigma_k}\|{\bm{w}}^k-\sigma_k {\bm{A}}^{\top}{\bm{u}}\|^2\\
		&-\frac{1}{2\sigma_k}\|{\bm{w}}^k\|^2 +\frac{1}{2\gamma_k}(v^k-\gamma_k {\bm{u}}^{\top}{\bf{1}}_m)^2-\frac{1}{2\gamma_k}(v^k)^2,\ \forall {\bm{u}}\in {\rm{dom}}h^*.
	\end{split}
\end{equation}
In fact, one easily deduces that if ${\bm{u}}^{k+1} \approx \underset{{\bm{u}}\in{{\rm{dom}}h^*}}{\rm argmin} \psi_k({\bm{u}}),$ then 
$$
({\bm{x}}^{k+1},v^{k+1},{\bm{y}}^{k+1})=\mathop {\arg\min }\limits_{{\bm{w}},v,{\bm{y}}}L({\bm{w}},v,{\bm{y}};{\bm{u}}^{k+1}),
$$ 
that is, ${\bm{w}}^{k+1}$ and $v^{k+1}$ in \eqref{2.5} have the following closed-form expressions respectively:
$$
{\bm{w}}^{k+1}={\rm Prox}_{\sigma_k \lambda \|\cdot\|_1}({\bm{w}}^k-\sigma_k{\bm{A}}^{\top}{\bm{u}}^{k+1}),\ v^{k+1}=v^k-\gamma_k({\bm{u}}^{k+1})^{\top}\boldsymbol{1}_m.
$$
In addition, the auxiliary variable ${\bm{y}}^{k+1}$ admits the closed-form expression: 
\begin{equation}\label{aux}
	{\bm{y}}^{k+1}_i=\frac{1}{{\bm{b}}_i}\log(-\frac{{\bm{b}}_i}{m{\bm{u}}^{k+1}_i}-1),\ \forall i=1,...,m,
\end{equation}
which implies that  $\nabla h({\bm{y}})-{\bm{u}}={\bm{0}}$ always holds  in each iteration.
\subsection{The PPDNA and Its Convergence}
In this subsection,	the  framework of a dual Newton method based proximal point algorithm and the results of its convergence are briefly described below.

\begin{algorithm}
		\caption{ {(PPDNA)} Dual Newton method based proximal point algorithm for \eqref{2.1} } \label{al1}
	\begin{algorithmic}[1]
		\Require $\sigma_0>0$, $\gamma_0>0$, $({\bm{w}}^0,v^0)\in\mathbb{R}^n\times\mathbb{R} , {\bm{u}}^0\in{{\rm{dom}}h^*}$. Set $k=0$. 
		\State Approximately compute
		\begin{equation}\label{sub}
			{\bm{u}}^{k+1} \approx \underset{{\bm{u}}\in{{\rm{dom}}h^*}}{\rm argmin} \psi_k({\bm{u}}),
		\end{equation}
		to satisfy the stopping criteria \eqref{eqn1} and \eqref{eqn2}.
		\State Compute ${\bm{w}}^{k+1}={\rm Prox}_{\sigma_k \lambda \|\cdot\|_1}({\bm{w}}^k-\sigma_k{\bm{A}}^{\top}{\bm{u}}^{k+1}),\ v^{k+1}=v^k-\gamma_k({\bm{u}}^{k+1})^{\top}\bm{1}_m.$
		\State Update $\sigma_{k+1} \uparrow \sigma_\infty\leq \infty,\gamma_{k+1} \uparrow \gamma_\infty\leq \infty$, $k\leftarrow k+1,$ and go to step $1$.
	\end{algorithmic}
\end{algorithm}

In order to facilitate the convergence analysis of the PPDNA, we first define the following function $f_k(\cdot)$: 
\begin{equation}\label{fk}
	f_k({\bm{w}},v):=h({\bm{Aw}}+v{\bf{1}}_m)+\lambda p({\bm{w}})+\frac{1}{2\sigma_k}\| ({\bm{w}};v)             
	-({\bm{w}}^k;v^k)\|_{\mathcal{M}_k}^2,
\end{equation}
where ${{\mathcal{M}}}_k={\rm{Diag}}({\bm{1}}_n;\sigma_k{\gamma_k}^{-1})\in\mathbb{R}^{(n+1)\times(n+1)}$.
Then, problem \eqref{2.5} can be equivalently rewritten as
\begin{equation}\label{e2.5}
	({\bm{w}}^{k+1};v^{k+1})
	\approx\mathcal{P}_k({\bm{w}}^k,v^k):=\mathop {\arg\min }\limits_{{\bm{w}},v} \left\{f_k({\bm{w}},v)\right\}.
\end{equation}
Additionally, we  make the following assumption \cite{L.S.T.2020}.
\begin{assumption}\label{assumption1}
	The sequences $\{\sigma_k\gamma_k^{-1}\}$ and $\{\sigma_k\}$ of positive real numbers are both bounded away from $0$. The sequence $\{\mathcal{M}_k\}$ of self-adjoint positive definite linear operators satisfies the following conditions 
	\begin{equation}\nonumber
		\mathcal{M}_{k}\succeq\mathcal{M}_{k+1},\ \mathcal{M}_k\succeq\min\{1,\sigma_\infty\gamma_\infty^{-1}\}{\bm{I}}_{n+1},\ \forall k\geq 0.
	\end{equation}
\end{assumption} 

As studied in \cite{R.1976}, the stopping criteria for step $1$ in the PPDNA algorithm are as follows:
\begin{align}
	\| ({\bm{w}}^{k+1};v^{k+1})-\mathcal{P}_k({\bm{w}}^k,v^k) \|_{\mathcal{M}_{k}} & \leq \epsilon_k,\ \epsilon_k\geq 0,\ \sum_{k=0}^{\infty}\epsilon_k<\infty, \label{eqn1}\tag{A}\\
	\| ({\bm{w}}^{k+1};v^{k+1}) -\mathcal{P}_k({\bm{w}}^k,v^k) \|_{\mathcal{M}_{k}}& \leq \delta_k\|({\bm{w}}^{k+1};v^{k+1}) -({\bm{w}}^{k};v^{k}) \|_{\mathcal{M}_{k}}, \label{eqn2}\tag{B}
\end{align}
where the sequence $\{\delta_k\}$ satisfies $0\leq \delta_k<1$ and $\ \sum_{k=0}^{\infty}\delta_k<\infty$. The above stopping criteria are generally not achievable in practice due to the difficulty of computing $\mathcal{P}_k({\bm{w}}^k,v^k)$. However, it follows from \cite[Proposition $3$]{L.S.T.2019} that the above stopping criteria can be replaced by the following implementable criteria.
\begin{proposition}\label{prop:criteria}
	In Algorithm \ref{al1}, the stopping criteria \eqref{eqn1} and \eqref{eqn2} in solving problem \eqref{sub} can be achieved by the following  implementable criteria:
	\begin{align}
		f_k({\bm{w}}^{k+1},v^{k+1})+\psi_k({\bm{u}}^{k+1}) & \leq\frac{\epsilon_k^2}{2\sigma_k},\ \epsilon_k\geq 0,\ \sum_{k=0}^{\infty}\epsilon_k<\infty, \label{eqn3}\tag{A'}\\
		f_k({\bm{w}}^{k+1},v^{k+1})+\psi_k({\bm{u}}^{k+1}) & \leq \frac{\delta_k^2}{2\sigma_k}\|({\bm{w}}^{k+1};v^{k+1}) -({\bm{w}}^{k};v^{k}) \|_{\mathcal{M}_{k}}^2,\sum_{k=0}^{\infty}\delta_k<\infty,\label{eqn4}\tag{B'}
	\end{align}
	where for $k=1,2,\ldots$, $0\leq \delta_k<1$ and $f_k(\cdot)$ and $\psi_k(\cdot)$ are defined in \eqref{fk} and \eqref{2.8} respectively.
\end{proposition}
\begin{proof}
	We can demonstrate this conclusion by following the proof in  the literature \cite[Proposition $3$]{L.S.T.2019}. Here we omit the details of the proof. 
\end{proof}

Since local convergence of the PPDNA relies on the assumption of error bounds, some results on error bounds should be given before we analyze the convergence of the PPDNA. The maximal monotone operator \cite{R.R.1976} of the objective function $f$ in problem \eqref{2.1} is defined by
\begin{equation}\label{gradf}
	\mathcal{T}_f({\bm{w}},v):= \partial f({\bm{w}},v)=\left [              
	\begin{array}{c}   
		{\bm{A}}^{\top}\nabla h({\bm{Aw}}+v{\bm{1}}_m)+\partial \lambda p({\bm{w}}) \\ [5pt]
		{\bm{1}}_m^{\top}\nabla h({\bm{Aw}}+v{\bm{1}}_m)
	\end{array}
	\right ].
\end{equation}
	\begin{theorem}
		The error bound condition holds for problem \eqref{2.1}, i.e.,  $\forall ({\bm{w}},v)\in\mathbb{R}^n\times \mathbb{R}\ {\rm{satisfying}} \ {\rm{dist}}(({\bm{w}},v),\Omega_{\lambda})\leq r$, we have
		\begin{equation}\label{2.11}
			{\rm{dist}}(({\bm{w}},v),\Omega_{\lambda})\leq \kappa {\rm{dist}}({\bm{0}},\mathcal{T}_f({\bm{w}},v)).
		\end{equation}
	\end{theorem}
	\begin{proof}
		For problem \eqref{2.1}, it is easy to see that $h(\cdot)$ is continuously differentiable on ${\rm{dom}}h$, strongly convex on any compact convex set $\mathcal{V}\subseteq {\rm{dom}}h$ and its gradient $\nabla h$ is Lipschitz continuous on $\mathcal{V}$. Moreover, $p(\cdot)$ is a polyhedral convex regularizer. For given $r>0$, define
		$$
		\bar{\Omega}_{r}:=\{({\bm{w}},v)\in \mathbb{R}^n\times \mathbb{R}\ |\ {\rm{dist}}(({\bm{w}},v),\Omega_{\lambda})\leq r\}.
		$$
		Since $\Omega_{\lambda}$ is compact, it can be found that $\bar{\Omega}_{r}$ is compact and thus $\epsilon_f=\max_{({\bm{w}},v)\in\bar{\Omega}_{r}}f({\bm{w}},v)$ is finite. By virtue of the properties of $h(\cdot)$ and $p(\cdot)$, combining with \cite[Proposition 6]{Z.SAM.2017}, one can obtain that $\mathcal{T}_f$ satisfies the error bound with the proximal map-based residual function, i.e., for the above $\epsilon_f$, there exist constants $\kappa_0,\epsilon >0$ such that $\forall ({\bm{w}},v)\in\mathbb{R}^n\times \mathbb{R} \ {\rm{with}}\  f({\bm{w}},v)\leq \epsilon_f$ and $\|{R}_{\lambda}({\bm{w}},v)\|\leq \epsilon$, it holds that
		\begin{equation*}
			{\rm{dist}}(({\bm{w}},v),\Omega_{\lambda})\leq \kappa_0\|{R}_{\lambda}({\bm{w}},v)\|.
		\end{equation*}
		
		The remaining proof  follows from the work of Li et al. \cite{L.S.T.2019} without any difficulty. Here we omit the details of the proof.
	\end{proof}
	
	From the work of Li et al. \cite{L.S.T.2020}, we readily get the following results for the global and local convergence of the PPDNA.
	\begin{theorem}
		(1) Let $\{({\bm{w}}^k,v^k)\}$ be the sequence generated by the PPDNA with stopping criterion \eqref{eqn1}. Then $\{({\bm{w}}^k,v^k)\}$ is bounded and 
		\begin{equation}\nonumber
			{\rm{dist}}_{\mathcal{M}_{k+1}}(({\bm{w}}^{k+1},v^{k+1}),\Omega_{\lambda})\leq {\rm{dist}}_{\mathcal{M}_{k}}(({\bm{w}}^{k},v^{k}),\Omega_{\lambda})+\epsilon_k,\ \forall k\geq 0.
		\end{equation}
		In addition, $\{({\bm{w}}^k,v^k)\}$ converges to the optimal point $\{({\bm{w}}^*,v^*)\}$ of problem \eqref{2.1} such that ${\bm{0}}\in \mathcal{T}_f({\bm{w}}^*,v^*)$.
		
		(2) Let $r:=\sum_{i=0}^{\infty}\epsilon_k+{\rm{dist}}_{\mathcal{M}_{0}}(({\bm{w}}^{0},v^{0}),\Omega_{\lambda})$. Then, for this $r>0$, there exists a constant $\kappa>0$ such that $\mathcal{T}_f$ satisfies the error bound condition \eqref{2.11}. Suppose that $\{({\bm{w}}^{k},v^{k})\}$ is the sequence generated by the PPDNA with the stopping criteria \eqref{eqn1} and \eqref{eqn2} with nondecreasing $\{\sigma_k\}$. Then it holds that for all $k\geq 0$, 
		\begin{equation*}
			{\rm{dist}}_{\mathcal{M}_{k+1}}(({\bm{w}}^{k+1},v^{k+1}),\Omega_{\lambda})\leq \mu_k{\rm{dist}}_{\mathcal{M}_{k}}(({\bm{w}}^{k},v^{k}),\Omega_{\lambda}),
		\end{equation*}
		where 
		$$
		\mu_k=\frac{1}{1-\delta_k}\left[\delta_k+\frac{(1+\delta_k)\kappa\lambda_{\max}(\mathcal{M}_k)}{\sqrt{\sigma_k^2+\kappa^2\lambda_{\max}^2(\mathcal{M}_k)}}\right]
		$$
		and 
		$$
		\mathop{\lim\sup}\limits_{k\to\infty}\mu_k=\mu_{\infty}=\frac{\kappa\lambda_{\infty}}{\sqrt{\sigma_\infty^2+\kappa^2\lambda_{\infty}^2}}<1
		$$
		with $\lambda_\infty=\mathop{\lim\sup}\limits_{k\to\infty}\lambda_{\max}(\mathcal{M}_k).$ In addition, it holds that for all $k\geq0$,
		$$
		{\rm{dist}}(({\bm{w}}^{k+1},v^{k+1}),\Omega_{\lambda})\leq \frac{\mu_k}{\sqrt{\lambda_{\min}(\mathcal{M}_{k+1})}}{\rm{dist}}_{\mathcal{M}_{k}}(({\bm{w}}^{k},v^{k}),\Omega_{\lambda}).
		$$
	\end{theorem}
	\begin{proof}
		The proof of item (1) follows from \cite[Proposition 2.3]{L.S.T.2020}. Combining the error bound \eqref{2.11} of Proposition \ref{prop:criteria}  with  \cite[Proposition 2.5]{L.S.T.2020}, we complete the proof of item (2).
	\end{proof}

	\subsection{A Semismooth Newton Algorithm for Solving Subproblems}
	In this subsection, we focus on  an efficient semismooth Newton algorithm \cite{K.1988,Q.S.1993,S.S.2002,M.R.1977} for solving the  subproblem \eqref{sub} in the PPDNA. Given $\sigma,\gamma>0$ and $(\tilde{\bm{w}},\tilde{v})\in\mathbb{R} ^n\times \mathbb{R},$ we aim to solve problem \eqref{sub}.  We define the function $\psi(\cdot)$ as 
	\begin{equation}\nonumber
		\begin{split}
			\psi({\bm{u}}):=&h^*({\bm{u}})-\frac{1}{\sigma_k}E_{\sigma\lambda p}(\tilde{\bm{w}}-\sigma{\bm{A}}^{\top}{\bm{u}})+\frac{1}{2\sigma}\|\tilde{\bm{w}}-\sigma_k {\bm{A}}^{\top}{\bm{u}}\|^2\\
			&-\frac{1}{2\sigma}\|\tilde{\bm{w}}\|^2 +\frac{1}{2\gamma}(\tilde{v}-\gamma {\bm{u}}^{\top}{\bf{1}}_m)^2-\frac{1}{2\gamma_k}\tilde{v}^2,\ \forall {\bm{u}}\in {\rm{dom}}h^*.
		\end{split}
	\end{equation}
	Since $\psi(\cdot)$ is strongly convex and continuously differentiable, problem \eqref{sub} has a unique optimal solution ${\bm{u}}^*\in{{\rm{dom}}h^*}$, which can be obtained by solving the following nonsmooth equations 
	\begin{equation*}
		\nabla\psi({\bm{u}})=\nabla h^*({\bm{u}})-{\bm{A}}{\rm{Prox}}_{\sigma \lambda \|\cdot\|_1}({\tilde{\bm{w}}}-\sigma {\bm{A}}^{\top}{\bm{u}})-(\tilde{v}-\gamma {\bm{u}}^{\top}{\bm 1}_m){\bm 1}_m={\bm{0}}, \forall {\bm{u}}\in {{\rm{dom}}h^*}.
	\end{equation*}
	Note that $\nabla h^*(\cdot)$ and ${\rm{Prox}}_{\sigma\lambda \|\cdot\|_1}(\cdot)$ are both Lipschitz continuous, one knows that the multifunction $\hat{\partial}^2\varphi(\cdot):\mathbb{R}^m \to \mathbb{R}^{m\times m}$ is well defined, which is given by
	\begin{equation*}
		\hat{\partial}^2\psi({\bm{u}}):=\nabla^2h^*({\bm{u}})+\sigma {\bm{A}}\partial{\rm{Prox}}_{\sigma \lambda \|\cdot\|_1}({\tilde{\bm{w}}}-\sigma {\bm{A}}^{\top}{\bm{u}}){\bm{A}}^{\top}+\gamma {\bm{1}}_m {\bm{1}}_m^{\top},\ \forall {\bm{u}}\in{{\rm{dom}}h^*},
	\end{equation*}
	where $\nabla^2h^*(\cdot)$ is defined in \eqref{def4} and $\partial{\rm{Prox}}_{\sigma \lambda \|\cdot\|_1}({\tilde{\bm{w}}}-\sigma {\bm{A}}^{\top}{\bm{u}})$ denotes the Clarke subdifferential \cite{CFH.1990} of the proximal mapping ${\rm{Prox}}_{\sigma \lambda \|\cdot\|_1}(\cdot)$ at ${\tilde{\bm{w}}}-\sigma {\bm{A}}^{\top}{\bm{u}}$. 
	Define ${\partial}^2\psi({\bm{u}})$ as  the Clarke generalized Jacobian of $\nabla\psi(\cdot)$ at ${\bm u}$. Then, it follows from \cite[Proposition 2.3.3 and Theorem 2.6.6]{CFH.1990} that  
	$$
	{\partial}^2\psi({\bm{u}}){\bm{d}}\subseteq\hat{\partial}^2\psi({\bm{u}}){\bm{d}},\ \forall {\bm{d}} \in \mathbb{R}^m.
	$$
	Let ${\bm{U}}\in\partial{\rm{Prox}}_{\sigma \lambda \|\cdot\|_1}({\tilde{\bm{w}}}-\sigma {\bm{A}}^{\top}{\bm{u}}) $. Then, we have
	\begin{equation*}
		{\bm{\mathcal{H}}}:=\nabla^2h^*({\bm{u}})+\sigma {\bm{A}}{\bm{U}}{\bm{A}}^{\top}+\gamma {\bf{1}}_m{\bf{1}}_m^{\top}\in \hat{\partial}^2\psi({\bm{u}}).
	\end{equation*}
	It can be found that $\nabla^2h^*(\cdot)$ is symmetric positive definite on ${{\rm{dom}}h^*}$ and hence  ${\bm{\mathcal{H}}}$ is symmetric positive definite on ${{\rm{dom}}h^*}$.
	
	Since $\nabla h^*(\cdot)$ and ${\rm{Prox}}_{\sigma \lambda \|\cdot\|_1}(\cdot)$ are both strongly semismooth functions, one knows that $\nabla \psi(\cdot)$ is strongly semismooth. Thus we shall design the semismooth Newton algorithm to solve problem \eqref{sub}, which is shown in Algorithm \ref{al2}.
	\begin{algorithm}[H]
		\caption{ {({\sc Ssn}) A semismooth Newton algorithm for solving problem \eqref{sub} }}\label{al2}	
		\begin{algorithmic}[1]
			\Require
			$\mu\in(0,1/2),\bar \tau\in (0,1)$ and $\bar\eta,\eta\in(0,1)$. Given an initial point ${{\bm{u}}^0\in{{\rm{dom}}h^*}}$, and set  $j=0$.
			
			\Ensure The approximate solution ${\bm{u}}^{j+1}$ of problem \eqref{sub}.
			
			\State Choose ${\bm{\mathcal{H}}}_j\in{\hat{\partial}}^2\psi({\bm{u}}^{j})$. Solve the following linear system 
			\begin{equation}\label{newton}
				\begin{split}
					{\bm{\mathcal{H}}}_j{\bm{d}}= -\nabla\psi({\bm{u}}^{j})
				\end{split}
			\end{equation}
			by the direct method or the conjugate gradient method such that the approximate solution ${\bm{d}}^j\in\mathbb{R}^m$ satisfies
			$$
			\|{\bm{\mathcal{H}}}_j{\bm{d}}^j+\nabla\psi({\bm{u}}^{j})\|\leq \min(\bar \eta,\|\nabla\psi({\bm{u}}^{j})\|^{1+\bar\tau}).
			$$
			
			\State Set $\alpha_j=\eta^{\bar{c}_j}$, where $\bar{c}_j$ is the smallest nonnegative integer $\bar{c}$ satisfying 
			$$
			{\bm{u}}^{j}+\eta^{\bar{c}}{\bm{d}}^{j}\in{{\rm{dom}}h^*},\ \psi({\bm{u}}^{j}+\eta^{\bar{c}}{\bm{d}}^{j})\leq\psi({\bm{u}}^{j})+\mu\eta^{\bar{c}}\left\langle\nabla\psi({\bm{u}}^{j}),{\bm{d}}^{j}\right\rangle.
			$$
			
			\State Update ${\bm{u}}^{j+1}={\bm{u}}^{j}+\alpha_j{\bm{d}}^{j},\ j\gets j+1$, and go to step $1$.
		\end{algorithmic}
	\end{algorithm}
	By virtue of \cite[Theorem 3.4 and 3.5]{Z.S.T.2010}, it is easy to obtain the following results on the convergence of the {\sc Ssn} algorithm.
	
	\begin{theorem}
		The sequence $\{{\bm{u}}^{j}\}$ generated by the {\sc Ssn} algorithm converges to the unique optimal solution ${\hat{\bm{u}}}$ of problem \eqref{sub}. Moreover, the rate  of convergence is at least superlinear with
		$$
		\|{\bm{u}}^{j+1}-{\hat{\bm{u}}}\|=\mathcal{O}(\|{\bm{u}}^{j}-{\hat{\bm{u}}}\|^{1+\bar{\tau}}),
		$$
		where $\bar{\tau}$ is given in the {\sc Ssn} algorithm.
	\end{theorem}
	
	\subsection{Efficient Implementations of the {\sc Ssn} Algorithm}
	The computational cost of solving Newton linear equation \eqref{newton} in the {\sc Ssn} algorithm is expensive, especially when the dimension of the problem is large. In order to improve the efficiency of the {\sc Ssn} algorithm, in this subsection we shall further analyze the linear system and exploit some special structures to effectively reduce its computational cost.
	
	Given $({\tilde{\bm{w}}},{\bm{u}})\in \mathbb{R}^n\times\mathbb{R}^m$ and $\sigma,\gamma>0$,  we need to solve the following linear Newton system:
	\begin{equation}\label{2.15}
		(\nabla^2h^*({\bm{u}})+\sigma {\bm{AUA}}^{\top}+\gamma {\bm{1}}_m {\bm{1}}_m^{\top}){\bm{d}}=-\nabla\psi({\bm{u}}),
	\end{equation}
	where $\nabla^2h^*({\bm{u}})$ is defined in \eqref{def4} and ${\bm{U}}\in\partial{\rm{Prox}}_{\sigma \lambda \|\cdot\|_1}({\tilde{\bm{w}}}-\sigma {\bm{A}}^{\top}{\bm{u}})$. Since $\nabla^2h^*({\bm{u}})$ is a positive definite diagonal matrix, we use $\nabla^2h^*({\bm{u}})={\bm{LL}}^{\top}$ to denote the Cholesky decomposition of $\nabla^2h^*({\bm{u}})$, where $L$ is a positive definite diagonal matrix. Then, we reformulate the equations \eqref{2.15} equivalently as 
	\begin{equation}\label{2.16}
		({\bm {I}}_m+\sigma({\bm{L}}^{-1}{\bm{A}}){\bm{U}}({\bm{L}}^{-1}{\bm{A}})^{\top}+\gamma({\bm{L}}^{-1}{\bm{1}}_m)({\bm{L}}^{-1}{\bm{1}}_m)^{\top})({\bm{L}}^{\top}{\bm{d}})=-{\bm{L}}^{-1}\nabla \psi({\bm{u}}).
	\end{equation}
	Obviously, by the property of $\nabla^2h^*({\bm{u}})$, one knows that the cost of computing ${\bm{L}}$ and ${\bm{L}}^{-1}$ are both very low and almost negligible. For the convenience of later analysis, the linear system \eqref{2.16} can be rewritten as:
	\begin{equation}\label{2.17}
		({\bm{I}}_m+\sigma{\hat{\bm{A}}}{\bm{U}}{\hat{\bm{A}}}^{\top}+\gamma{\hat{\bm{1}}}_m{\hat{\bm{1}}}_m^{\top}){\hat{\bm{d}}}=-\nabla\hat{\psi}({\bm{u}}),
	\end{equation}
	where ${\hat{\bm{A}}}={\bm{L}}^{-1}{\bm{A}}\in\mathbb{R}^{m\times n},{\hat{\bm{1}}}_m={\bm{L}}^{-1}{\bm{1}_m}\in\mathbb{R}^{m}, {\hat{\bm{d}}}={\bm{L}}^{\top}{\bm{{d}}}\in\mathbb{R}^{m},\nabla\hat{\psi}({\bm{u}})={\bm{L}}^{-1}\nabla\psi({\bm{u}})\in\mathbb{R}^{m}$. The cost of naively computing ${\hat{\bm{A}}{\bm{U}}\hat{\bm{A}}}^{\top}$ is $\mathcal{O}(m^2n)$, while for any given ${\hat{\bm{d}}}\in \mathbb{R}^m$, the cost of computing ${\hat{\bm{A}}{\bm{U}}\hat{\bm{A}}}^{\top}{\hat{\bm{d}}}$ is $\mathcal{O}(mn)$. As the scale of problem \eqref{sub} continues to expand, the expensive computational cost makes some commonly used algorithms such as the Cholesky decomposition method or conjugate gradient method unsuitable for solving this linear system \eqref{2.17}. Thus we consider fully exploiting the second-order sparsity of problem \eqref{sub} to reduce the computational cost of linear systems. In our implementations, we choose the diagonal matrix ${\bm{U}}={\rm{Diag}}({\bm{\xi}})$, where the diagonal elements of the matrix ${\bm{U}}$ are given by
	\begin{equation*}
		{\bm{\xi}}_i:=\left\{
		\begin{aligned}
			0, &\quad  \ |{\tilde{\bm{w}}}-\sigma {\bm{A}}^{\top}{\bm{u}}|_i\leq \sigma\lambda, \\
			1, &\quad  \ |{\tilde{\bm{w}}}-\sigma{\bm{A}}^{\top}{\bm{u}}|_i> \sigma\lambda , \\
		\end{aligned} \right.\ \ i=1,\ldots,n.
	\end{equation*}
	Then, one can readily obtain that ${\bm{U}}\in\partial{\rm{Prox}}_{\sigma \lambda \|\cdot\|_1}({\tilde{\bm{w}}}-\sigma {\bm{A}}^{\top}{\bm{u}})$. Here we define the index set $\mathcal{J}:=\{j\ | \ {\bm{\xi}}_j=1,\ j=1,\ldots,n \}$.  Based on the special $0$-$1$ structure of ${\bm{U}}$, it holds that
	\begin{equation}\label{AA}
		{\hat{\bm{A}}}{\bm{U}}\hat{{\bm{A}}}^{\top}={\hat{\bm{A}}}_{\mathcal{J}}{\hat{\bm{A}}}_{\mathcal{J}}^{\top},
	\end{equation}
	where ${\hat{\bm{A}}}_{\mathcal{J}}\in\mathbb{R}^{m\times r}$ denotes the matrix consisting of the columns of ${\hat{\bm{A}}}$ indexed by ${\mathcal{J}}$ and $r$ denotes the cardinality of the set $\mathcal{J}$.  By this equivalent transformation \eqref{AA}, the computational cost of ${\hat{\bm{A}}}{\bm{U}}\hat{{\bm{A}}}^{\top}$ and ${\hat{\bm{A}}}{\bm{U}}\hat{{\bm{A}}}^{\top}{\bm{d}}$ for a given vector ${\bm{d}}$  can be reduced to $\mathcal{O}(m^2r)$ and $\mathcal{O}(mr)$. In particular, the total computational cost of solving the linear system by using the  Cholesky decomposition is reduced from $\mathcal{O}(m^2(m+n+1))$ to $\mathcal{O}(m^2(m+r+1))$. Furthermore, if $r$ is much smaller than $m$, we can use the Sherman-Morrison-Woodbury formula \cite{G.VCF.2013} to calculate the inverse of $({\bm{I}}_m+\sigma{\hat{\bm{A}}}{\bm{U}}\hat{{\bm{A}}}^{\top}+\gamma\hat{\bm{1}}_m\hat{\bm{1}}_m^{\top})$, which makes the computation cheaper. Let ${\bm{W}}=[{\hat{\bm{A}}}_{\mathcal{J}},\sqrt{{\gamma}/{\sigma}}{\bm{\hat{1}}}_m]\in \mathbb{R}^{m\times(r+1)}$,  we have 
	$$\begin{aligned}
		({\bm{I}}_m+\sigma{\hat{\bm{A}}}{\bm{U}}\hat{\bm A}^{\top}+\gamma{\hat{\bm{1}}}_m{\hat{\bm{1}}}_m^{\top})^{-1}&=({\bm{I}}_m+\sigma {\bm{WW}}^{\top})^{-1}\\
		&={\bm{I}}_m-{\bm{W}}(\sigma^{-1} {\bm{I}}_{r+1}+{\bm{W}}^{\top}{\bm{W}})^{-1}{\bm{W}}^{\top}.\end{aligned}
	$$
	From the above analysis, we only need to factorize a $(r+1)\times (r+1)$ matrix instead of a $m\times m$ matrix. Thus, the total computational cost of solving the linear system by the Cholesky decomposition  is reduced from  $\mathcal{O}(m^2(m+r+1))$ to $\mathcal{O}((r+1)^2(m+r+1))$.
	
	As a result, we greatly reduce the computational cost of  the Newton linear system \eqref{2.17} by exploiting the special structure of $\nabla^2h^*({\bm{u}})+\sigma {\bm{AUA}}^{\top}+\gamma {\bm{1}}_m {\bm{1}}_m^{\top}$.
	
	\section{An Adaptive Sieving Strategy}\label{sec:3}
	In this section, we develop an adaptive sieving strategy \cite{L.Y.S.2020,Y.C.S.2021} for generating solution paths of problem \eqref{2.1}. The main idea of the adaptive sieving strategy is to reduce the number of variables by using the constraints of the index set and obtain the optimal solution of the original problem \eqref{2.1} by solving the problem with smaller dimension. This strategy can  improve the efficiency of the algorithm by reducing the dimension of the problem.
	
	Our algorithm combines an adaptive sieving strategy with the PPDNA. We reduce the dimension of the problem via the adaptive sieving strategy and then apply the PPDNA to solve the problem with smaller dimension. The specific process of the adaptive sieving strategy is described in Algorithm \ref{al3}. For a given sequence of hyper-parameter, we first solve problem \eqref{3.2} with a larger $\lambda^1$ under the constraints of a reasonable initial index set, and update the constrained index set according to the KKT residuals until an approximate optimal solution that satisfies the KKT conditions is obtained. Then, for the next smaller $\lambda^2$, we solve problem \eqref{3.2} under the constraints of the index set corresponding to the non-zero solution of the previous problem. By analogy, we can get a solution path of problem \eqref{2.1} by performing this process on the following hyper-parameters.
	
	Next, we shall explain problems \eqref{3.2} and \eqref{3.3} in Algorithm \ref{al3}, and then analyze the construction of $J^{l+1}(\lambda^i)$ in step $6$ of Algorithm \ref{al3}. In Algorithm \ref{al3}, the KKT residual functions ${Res}_\lambda(\cdot)$ and ${Res}^{I^0}_{\lambda^i}(\cdot)$  are respectively defined by
	\begin{equation*}
		\begin{split}
			{Res}_\lambda({\bm{w}},v,{\bm{y}},{\bm{u}}):=\max\{&\|\nabla h({\bm{y}})-{\bm{u}}\|,\|{\bm{w}}-{\rm{Prox}}_{\lambda \|\cdot\|_1}({\bm{w}}-{\bm{A}}^{\top}{\bm{u}})\|, \\
			&\|{\bm{y}}-{\bm{Aw}}-v{\bm{1}}_m\|, |{\bm{u}}^{\top}{\bm{1}}_m|\}.\\
			{Res}^{I^0}_{\lambda^i}({\bm{z}},v,{\bm{y}},{\bm{u}}):=\max\{&\|\nabla h({\bm{y}})-{\bm{u}}\|,\|{\bm{z}}-{\rm{Prox}}_{\lambda^i \|\cdot\|_1}({\bm{z}}-{({\bm{A}}_{{I^0}(\lambda^i)})}^{\top}{\bm{u}})\|, \\
			&\|{\bm{y}}-{\bm{A}}_{{I^0}(\lambda^i)}{\bm{z}}-v{\bm{1}}_m\|, |{\bm{u}}^{\top}{\bm{1}}_m|\}.
		\end{split}
	\end{equation*}
	If  $({\bm{w}}^*,v^*,{\bm{y}}^*,{\bm{u}}^*)$ satisfies ${Res}_\lambda({\bm{w}}^*,v^*,{\bm{y}}^*,{\bm{u}}^*)\leq\epsilon$, we can also accept $ ({\bm{w}}^*,v^*)$  as the approximate optimal solution to the problem \eqref{2.1}.
	Taking problem \eqref{3.2} in Algorithm \ref{al3} as an example, in fact, we consider the following constrained optimization problem:
	\begin{equation}\label{3.4}
		\begin{split}
			\mathop{\min}\limits_{{\bm{w}}\in\mathbb{R}^n,v\in\mathbb{R}} &  h({\bm{Aw}}+v{\bm{1}}_m)+\lambda^i\|{\bm{w}}\|_1\\
			\  \ \mbox{s.t.}  \quad & {\bm{w}}_{\bar{I}^{0}(\lambda^i)}={\bm{0}}.
		\end{split}
	\end{equation}
	
	With the constraint $ {\bm{w}}_{\bar{I}^{0}(\lambda^i)}={\bm{0}}$, problem \eqref{3.4} can be transformed into the following problem with smaller size:
	\begin{equation}\label{3.5}
		\begin{split}
			\mathop{\min}\limits_{{\bm{z}}\in\mathbb{R}^{|{I}^{0}(\lambda^i)|},v\in\mathbb{R}}&   h({\bm{A}}_{{I}^{0}(\lambda^i)}{\bm{z}}+v{\bm{1_m}})+\lambda^i\|{\bm{z}}\|_1,
		\end{split}
	\end{equation}
	where ${\bm{A}}_{{I}^{0}(\lambda^i)}$ is the matrix consisting of the columns of ${\bm{A}}$ indexed by $ {{I}^{0}(\lambda^i)}$. 
	Since the problem \eqref{3.5} is of the same form as problem \eqref{2.1}, we can apply the efficient PPDNA to solve problem \eqref{3.5}. Obviously, it can greatly reduce the cost of computation and save storage space. The approximate optimal solution ${\bm{z}}^{0}(\lambda^i)$ of problem \eqref{3.5} can be extended to the approximate optimal solution ${\bm{w}}^{0}(\lambda^i)$ of problem \eqref{3.4} as follows:
	\begin{equation*}
		{\bm{w}}^{0}(\lambda^i)_{I^0(\lambda^i)}={\bm{z}}^{0}(\lambda^i),\ {\bm{w}}^{0}(\lambda^i)_{\bar{I}^0(\lambda^i)}={\bm{0}}.
	\end{equation*}
	
	\begin{algorithm}[p]
		\caption{ Adaptive sieving strategy for solving problem \eqref{2.1}}	\label{al3}
		\begin{algorithmic}[1]
			\Require Given a sequence of hyper-parameter: $\lambda^1>\ldots>\lambda^t>0$, and  tolerance $\epsilon>0$.
			
			\Ensure A solution path: $({\bm{w}}^*(\lambda^1),v^*(\lambda^1)),\ldots, ({\bm{w}}^*(\lambda^t),v^*(\lambda^t)).$
			
			\State {\bf{Initialization:}} 
			Generate an initial index set $ I^*(\lambda^0)$ by a screening rule.
			
			\For{$i=1,2,\ldots,t$} 
			
			\State Let $I^0(\lambda^i)=I^*(\lambda^{i-1}).$ Solving
			\begin{equation}\label{3.2}
				\mathop {\min }\limits_{{\bm{z}}\in\mathbb{R}^{|I^0(\lambda^i)|},v\in\mathbb{R}} \left\{h({\bm{A}}_{I^0(\lambda^i)}{\bm{z}}+v{\bf{1}}_m)+\lambda^i\|{\bm{z}}\|_1\right\}
			\end{equation}
			such that ${Res}^{I^0}_{\lambda^i}({\bm{z}}^{0}(\lambda^i),v^{0}(\lambda^i),{\bm{y}}^{0}(\lambda^i),{\bm{u}}^{0}(\lambda^i))\leq\epsilon/\sqrt{2},$ where the approximate optimal solution $({\bm{z}}^{0}(\lambda^i),v^{0}(\lambda^i),{\bm{y}}^{0}(\lambda^i),{\bm{u}}^{0}(\lambda^i))$ can be obtained by PPDNA.  Extend  ${\bm{z}}^{0}(\lambda^i)$ to ${\bm{w}}^{0}(\lambda^i)$ by
			$$
			{\bm{w}}^{0}(\lambda^i)_{I^0(\lambda^i)}={\bm{z}}^{0}(\lambda^i),\ {\bm{w}}^{0}(\lambda^i)_{\bar{I}^0(\lambda^i)}={\bm{0}},
			$$
			where $\bar{I}^0(\lambda^i)$ denotes the complement of ${I}^0(\lambda^i)$.
			
			\State Compute ${Res}_{\lambda^i}({\bm{w}}^0(\lambda^i),v^0(\lambda^i),{\bm{y}}^0(\lambda^i),{\bm{u}}^0(\lambda^i))$ and set $l=0$.
			
			\While{${Res}_{\lambda^i}({\bm{w}}^l(\lambda^i),v^l(\lambda^i),{\bm{y}}^l(\lambda^i),{\bm{u}}^l(\lambda^i))>\epsilon$}
			
			\State Create $J^{l+1}(\lambda^i)$:
			$$
			J^{l+1}(\lambda^i)=\left\{j\in\bar{I}^l(\lambda^i)\ | \ -({\bm{A}}^{\top}{\bm{u}}^l(\lambda^i))_j\notin \left(\partial\lambda^i \|{\bm{w}}^l(\lambda^i)\|_1+ \frac{\epsilon}{\sqrt{2|\bar{I}^l(\lambda^i)|}}\mathcal{B}_{\infty}\right) _j \right\},
			$$
			where $\bar{I}^l(\lambda^i)$ denotes the complement of ${I}^l(\lambda^i)$ and $(\mathcal{D})_j$ represents the projection of the set $\mathcal{D}$ onto the $j$-th dimension. 
			
			\State Update $I^{l+1}\gets I^{l}\cup J^{l+1}$.
			
			\State Solve the following problem:
			\begin{equation}\label{3.3}
				\mathop {\min }\limits_{{\bm{z}}\in\mathbb{R}^{|I^{l+1}(\lambda^i)|},v\in\mathbb{R}} \left\{h({\bm{A}}_{I^{l+1}(\lambda^i)}{\bm{z}}+v{\bm{1}}_m)+\lambda^i\|{\bm{z}}\|_1\right\}
			\end{equation}
			such that ${Res}^{I^{l+1}}_{\lambda^i}({\bm{z}}^{l+1}(\lambda^i),v^{l+1}(\lambda^i),{\bm{y}}^{l+1}(\lambda^i),{\bm{u}}^{l+1}(\lambda^i))\leq\epsilon/\sqrt{2},$ where the approximate optimal solution $({\bm{z}}^{l+1}(\lambda^i),v^{l+1}(\lambda^i),{\bm{y}}^{l+1}(\lambda^i),{\bm{u}}^{l+1}(\lambda^i))$ is obtained by PPDNA. Extend ${\bm{z}}^{l+1}(\lambda^i)$ to ${\bm{w}}^{l+1}(\lambda^i)$ by 
			$$
			{\bm{w}}^{l+1}(\lambda^i)_{I^{l+1}(\lambda^i)}={\bm{z}}^{l+1}(\lambda^i),\ {\bm{w}}^{l+1}(\lambda^i)_{\bar{I}^{l+1}(\lambda^i)}={\bm{0}}.
			$$
			
			\State Compute ${Res}_{\lambda^i}({\bm{w}}^{l+1}(\lambda^i),v^{l+1}(\lambda^i),{\bm{y}}^{l+1}(\lambda^i),{\bm{u}}^{l+1}(\lambda^i))$ and set $l\gets {l+1}$.
			
			\EndWhile
			
			\State Set $({\bm{w}}^*(\lambda^i),v^*(\lambda^i),{\bm{y}}^*(\lambda^i),{\bm{u}}^*(\lambda^i))= ({\bm{w}}^l(\lambda^i),v^l(\lambda^i),{\bm{y}}^l(\lambda^i),{\bm{u}}^l(\lambda^i)),I^{*}(\lambda^i)= I^{l}(\lambda^i).$
			
			\EndFor
		\end{algorithmic}
	\end{algorithm}

	Now, we interpret the connection between the optimal solutions of problems \eqref{3.5} and \eqref{2.1}. Note that problem \eqref{3.5} is equivalent to the following one:
	\begin{equation}\label{eq3.5}
		\begin{aligned}
			\mathop{\min}\limits_{{\bm{z}}\in\mathbb{R}^{|{I}^{0}(\lambda^i)|},v\in\mathbb{R}, {\bm{y}}\in\mathbb{R}^m}&   h({\bm{y}})+\lambda^i\|{\bm{z}}\|_1\\
			\mbox{s.t.}\qquad \quad & {\bm{y}}={\bm{A}}_{{I}^{0}(\lambda^i)}{\bm{z}}+v{\bf{1}}_m,
		\end{aligned}
	\end{equation}
	one can obtain the KKT conditions of problem \eqref{eq3.5} as follows:
	\begin{equation}\label{3.6}
		\left\{
		\begin{aligned}
			&\nabla h({\bm{y}})-{\bm{u}}={\bm{0}},  \\
			&-({\bm{A}}_{I^0(\lambda^i)})^{\top}{\bm{u}}\in \partial\lambda^i \|{\bm{z}}\|_1,  \\
			&{\bm{y}}-{\bm{A}}_{I^0(\lambda^i)}{\bm{z}}-v{\bm{1}}_m={\bm{0}},\\
			&{\bm{u}}^{\top}{\bm{1}}_m=0.
		\end{aligned}\right.
	\end{equation}
	With ${\bm{w}}_{I^0(\lambda^i)}={\bm{z}},\ {\bm{w}}_{\bar{I}^0(\lambda^i)}={\bm{0}},$  the KKT conditions \eqref{3.6} can be rewritten as 
	\begin{equation}\label{3.7}
		\left\{
		\begin{aligned}
			&\nabla h({\bm{y}})-{\bm{u}}={\bm{0}},  \\
			&-({\bm{A}}^{\top}{\bm{u}})_{I^0(\lambda^i)}\in (\partial\lambda^i \|{\bm{w}}\|_1)_{I^0(\lambda^i)},  \\
			& {\bm{y}}-{\bm{A}}{\bm{w}}-v{\bm{1}}_m={\bm{0}},\\
			&{\bm{u}}^{\top}{\bf{1}}_m=0.
		\end{aligned}\right.
	\end{equation}
	Meanwhile, the equivalent KKT conditions \eqref{2.3} of problem \eqref{2.1} can be equivalently written as:
	\begin{equation}\label{3.8}
		\left\{
		\begin{aligned}
			&\nabla h({\bm{y}})-{\bm{u}}={\bm{0}},  \\
			&-({\bm{A}}^{\top}{\bm{u}})_{I^0(\lambda^i)}\in (\partial\lambda^i \|{\bm{w}}\|_1)_{I^0(\lambda^i)},  \\
			&-({\bm{A}}^{\top}{\bm{u}})_{\bar{I}^0(\lambda^i)}\in (\partial\lambda^i \|{\bm{w}}\|_1)_{\bar{I}^0(\lambda^i)},  \\
			& {\bm{y}}-{\bm{Aw}}-v{\bf{1}}_m={\bm{0}},\\
			&{\bm{u}}^{\top}{\bm{1}}_m=0.
		\end{aligned}\right.
	\end{equation}
	It is clear that the KKT systems  \eqref{3.7} and \eqref{3.8} differ by the condition 
	$$
	-({\bm{A}}^{\top}{\bm{u}})_{\bar{I}^0(\lambda^i)}\in (\partial\lambda^i \|{\bm{w}}\|_1)_{\bar{I}^0(\lambda^i)},
	$$ 
	which indicates  that the optimal solution to problem \eqref{3.5} indexed by the index set ${I}^0(\lambda^i)$ may not satisfy this condition. Therefore, in order to obtain the optimal solution to problem \eqref{2.1}, we need to use this condition to construct the index set  of step $6$ in Algorithm \ref{al3} and update the original index set ${I}^0(\lambda^i)$.
	
	Next, we further prove the connection between the residual function ${Res}_{\lambda^i}(\cdot)$  and the index set $J^{l+1}(\lambda^i)$ and establish the convergence result of Algorithm \ref{al3}.
	
	\begin{theorem}
		The solution path $({\bm{w}}^*(\lambda^1),v^*(\lambda^1)),\ldots,({\bm{w}}^*(\lambda^t),v^*(\lambda^t))$ obtained by Algorithm \ref{al3} are the approximate optimal solutions of problem $(P_{\lambda^1}),\ldots,(P_{\lambda^t})$, i.e., for the optimal solution  $({\bm{w}}^*(\lambda^i),v^*(\lambda^i))$ and  ${\bm{y}}^*(\lambda^i)$, there exists  ${\bm{u}}^*(\lambda^i)$ satisfying
		$$
		{Res}_{\lambda^i}({\bm{w}}^*(\lambda^i),v^*(\lambda^i),{\bm{y}}^*(\lambda^i),{\bm{u}}^*(\lambda^i))\leq\epsilon,\ \forall i=1,\ldots,t.
		$$
	\end{theorem}
	\begin{proof}
		For given $i\in\{1,\ldots,t\}$, we first prove that $J^{l+1}(\lambda^i)$ is not empty if ${Res}_{\lambda^i}({\bm{w}}^l(\lambda^i),v^l(\lambda^i),{\bm{y}}^l(\lambda^i),{\bm{u}}^l(\lambda^i))>\epsilon$. We prove it by contradiction. Suppose for the purpose of contrary that $J^{l+1}(\lambda^i)$ is the empty set, i.e., 
		$$
		-({\bm{A}}^{\top}{\bm{u}}^l(\lambda^i))_j\in (\partial\lambda^i \|{\bm{w}}^l(\lambda^i)\|_1)_j+( \frac{\epsilon}{\sqrt{2|\bar{I}^l(\lambda^i)|}}\mathcal{B}_{\infty})_j,\ \forall j\in\bar{I}^l(\lambda^i),
		$$
		which implies that there exists a vector $\hat{{\bm{\delta}}}^l_i\in\mathbb{R}^{|{\bar{I}^{l}(\lambda^i)}|}$ with $\|\hat{{\bm{\delta}}}^l_i\|_{\infty}\leq\frac{\epsilon}{\sqrt{2|\bar{I}^l(\lambda^i)|}}$ such that 
		\begin{equation*}
			-({\bm{A}}^{\top}{\bm{u}}^l(\lambda^i))_j+\hat{{\bm{\delta}}}^l_i\in (\partial\lambda^i \|{\bm{w}}^l(\lambda^i)\|_1)_j,\ \forall j\in\bar{I}^l(\lambda^i).
		\end{equation*}
		This indicates that 
		$$
		({\bm{w}}^l(\lambda^i))_{\bar{I}^l(\lambda^i)}={\rm{Prox}}_{\lambda^i \|\cdot\|_1}(({\bm{w}}^l(\lambda^i))_{{\bar{I}^l(\lambda^i)}}-\hat{{\bm{\delta}}}^l_i-({\bm{A}}^{\top}{\bm{u}}^l(\lambda^i))_{\bar{I}^l(\lambda^i)}).
		$$
		Thus, we have 
		\begin{equation}\label{IDE}
			\begin{split}
				&\quad\ \|({\bm{w}}^{l}(\lambda^i)-{\rm{Prox}}_{\lambda^i \|\cdot\|_1}({\bm{w}}^l(\lambda^i)-{\bm{A}}^{\top}{\bm{u}}^{l}(\lambda^i)))_{\bar{I}^{l}(\lambda^i)}\|\\
				&= \|{\rm{Prox}}_{\lambda^i \|\cdot\|_1}(({\bm{w}}^l(\lambda^i))_{{\bar{I}^l(\lambda^i)}}-\hat{{\bm{\delta}}}^l_i-({\bm{A}}^{\top}{\bm{u}}^l(\lambda^i))_{\bar{I}^l(\lambda^i)})\\
				&\quad \quad-{\rm{Prox}}_{\lambda^i \|\cdot\|_1}(({\bm{w}}^l(\lambda^i))_{{\bar{I}^l(\lambda^i)}}-({\bm{A}}^{\top}{\bm{u}}^l(\lambda^i))_{\bar{I}^l(\lambda^i)})\|\\
				&\leq \|\hat{{\bm{\delta}}}^l_i\|\leq \sqrt{\frac{\epsilon^2}{2|{\bar{I}^{l}(\lambda^i)}|}|{\bar{I}^{l}(\lambda^i)}|}=\epsilon/\sqrt{2}.
			\end{split}
		\end{equation}
		Note that $({\bm{z}}^{l}(\lambda^i),v^{l}(\lambda^i))$ is an approximate optimal solution to the following problem:
		\begin{equation*}
			\min \limits_{{\bm{z}}\in\mathbb{R}^{|I^{l}(\lambda^i)|},v\in\mathbb{R}} \left\{h({\bm{A}}_{I^{l}(\lambda^i)}{\bm{z}}+v{\bm{1}}_m)+\lambda^i\|{\bm{z}}\|_1\right\},
		\end{equation*}
		which  satisfies
		${Res}^{I^l}_{\lambda^i}({\bm{z}}^{l}(\lambda^i),v^{l}(\lambda^i),{\bm{y}}^{l}(\lambda^i),{\bm{u}}^{l}(\lambda^i))\leq\epsilon/\sqrt{2}$. Thus, one has 
		\begin{equation}\label{3.9}
			\left\{
			\begin{aligned}
				&\|\nabla h({\bm{y}}^{l}(\lambda^i))-{\bm{u}}^{l}(\lambda^i)\|\leq\epsilon/\sqrt{2},  \\
				& \|{\bm{z}}^{l}(\lambda^i)-{\rm{Prox}}_{\lambda^i \|\cdot\|_1}({\bm{z}}^l(\lambda^i)-({\bm{A}}_{I^{l}(\lambda^i)})^{\top}{\bm{u}}^{l}(\lambda^i))\|\leq \epsilon/\sqrt{2},\\
				&\|{\bm{y}}^{l}(\lambda^i)-{\bm{A}}_{I^{l}(\lambda^i)}{\bm{z}}^{l}(\lambda^i)-v^{l}(\lambda^i){\bm{1}}_m\|\leq\epsilon/\sqrt{2},\\
				&\|{\bm{u}}^{l}(\lambda^i)^{\top}{\bm{1}}_m\|\leq\epsilon/\sqrt{2}.
			\end{aligned}\right.
		\end{equation}
		Recall that ${\bm{w}}^{l}(\lambda^i)_{I^l(\lambda^i)}={\bm{z}}^{l}(\lambda^i),\ {\bm{w}}^{l}(\lambda^i)_{\bar{I}^{l}(\lambda^i)}={\bm{0}}$, we obtain that the second condition of \eqref{3.9} implies 
		\begin{equation*}
			\|({\bm{w}}^{l}(\lambda^i)-{\rm{Prox}}_{\lambda^i \|\cdot\|_1}({\bm{w}}^l(\lambda^i)-{\bm{A}}^{\top}{\bm{u}}^{l}(\lambda^i)))_{I^{l}(\lambda^i)}\|\leq \epsilon/\sqrt{2}.
		\end{equation*}
		Therefore, invoking \eqref{IDE}, we have
		\begin{equation*}
			\begin{aligned}
				&\ \quad \|{\bm{w}}^{l}(\lambda^i)-{\rm{Prox}}_{\lambda^i \|\cdot\|_1}({\bm{w}}^l(\lambda^i)-{\bm{A}}^{\top}{\bm{u}}^{l}(\lambda^i))\|^2\\
				&=\|({\bm{w}}^{l}(\lambda^i)-{\rm{Prox}}_{\lambda^i \|\cdot\|_1}({\bm{w}}^l(\lambda^i)-{\bm{A}}^{\top}{\bm{u}}^{l}(\lambda^i)))_{I^{l}(\lambda^i)}\|^2\\
				&\ \quad +\|({\bm{w}}^{l}(\lambda^i)-{\rm{Prox}}_{\lambda^i \|\cdot\|_1}({\bm{w}}^l(\lambda^i)-{\bm{A}}^{\top}{\bm{u}}^{l}(\lambda^i)))_{\bar{I}^{l}(\lambda^i)}\|^2\\
				&\leq \frac{\epsilon^2}{2}+\frac{\epsilon^2}{2}=\epsilon^2,
			\end{aligned}
		\end{equation*}
		which implies that
		\begin{equation*}
			\|{\bm{w}}^{l}(\lambda^i)-{\rm{Prox}}_{\lambda^i \|\cdot\|_1}({\bm{w}}^l(\lambda^i)-{\bm{A}}^{\top}{\bm{u}}^{l}(\lambda^i))\|\leq \epsilon.
		\end{equation*}
		Combining with  $({\bm{w}}^{l}(\lambda^i))_{\bar{I}^{l}(\lambda^i)}={\bm{0}}$, we obtain  
		$$
		\|{\bm{y}}^{l}(\lambda^i)-{\bm{A}}{\bm{w}}^{l}(\lambda^i)-v^{l}(\lambda^i){\bf{1}}_m\|=\|{\bm{y}}^{l}(\lambda^i)-{\bm{A}}_{I^{l}(\lambda^i)}{\bm{z}}^{l}(\lambda^i)-v^{l}(\lambda^i){\bm{1}}_m\|\leq\epsilon/\sqrt{2}.
		$$ 
		As a result, it holds that
		\begin{equation*}
			\left\{
			\begin{aligned}
				&\|\nabla h({\bm{y}}^{l}(\lambda^i))-{\bm{u}}^{l}(\lambda^i)\|\leq\epsilon/\sqrt{2}<\epsilon,  \\
				& \|{\bm{w}}^{l}(\lambda^i)-{\rm{Prox}}_{\lambda^i \|\cdot\|_1}({\bm{w}}^l(\lambda^i)-{\bm{A}}^{\top}{\bm{u}}^{l}(\lambda^i))\|\leq\epsilon, \\
				&\|{\bm{y}}^{l}(\lambda^i)-{\bm{A}}{\bm{w}}^{l}(\lambda^i)-v^{l}(\lambda^i){\bm{1}}_m\|\leq\epsilon/\sqrt{2}<\epsilon,\\
				&\|{\bm{u}}^{l}(\lambda^i)^{\top}{\bm{1}}_m\|\leq\epsilon/\sqrt{2}<\epsilon.
			\end{aligned}\right.
		\end{equation*}
		It means that $Res_{\lambda^i}({\bm{w}}^l(\lambda^i),v^l(\lambda^i),{\bm{y}}^l(\lambda^i),{\bm{u}}^l(\lambda^i))\leq\epsilon$ holds.
		Hence, we know that $J^{l+1}(\lambda^i)$ is not empty as long as $Res_{\lambda^i}({\bm{w}}^l(\lambda^i),v^l(\lambda^i),{\bm{y}}^l(\lambda^i),{\bm{u}}^l(\lambda^i))>\epsilon$. Since the total number of components of ${\bm{w}}$ is finite, Algorithm \ref{al3} must terminate after a finite number of iterations. Therefore, it follows directly from the convergence of Algorithm \ref{al3} that for $i=1,\dots,t$,
		$$
		{Res}_{\lambda^i}({\bm{w}}^*(\lambda^i),v^*(\lambda^i),{\bm{y}}^*(\lambda^i),{\bm{u}}^*(\lambda^i))= {Res}_{\lambda^i}({\bm{w}}^l(\lambda^i),v^l(\lambda^i),{\bm{y}}^l(\lambda^i)\leq \epsilon,
		$$
		which implies that the solution path $({\bm{w}}^*(\lambda^1),v^*(\lambda^1)),\ldots,({\bm{w}}^*(\lambda^t),v^*(\lambda^t))$ obtained by Algorithm \ref{al3} are approximate optimal solutions of problems $(P_{\lambda^1}),\ldots,(P_{\lambda^t})$, respectively. Here, we complete the proof.  
	\end{proof}

	\section{Numerical Experiments}\label{sec:4}
	In this section, we perform numerical experiments on the PPDNA and the AS strategy with the PPDNA for solving the $\ell_1$-regularized logistic regression problem \eqref{1.1} with bias term on random and real data sets. Firstly, we compare the PPDNA  with  the improved GLMNET (newGLMNET) method \cite{Y.H.L.2012}, the inexact regularized proximal Newton (IRPN) method \cite{Y.Z.S.2019},  and  the proximal Newton-type (PNT) method \cite{M.Y.Z.2022}. To demonstrate the numerical performance of the adaptive sieving strategy for generating solution paths of problem \eqref{1.1}, we test the AS strategy with the PPDNA, the newGLMNET method, the IRPN method, and the PNT method, respectively.  All our experiments are executed in MATLAB R2019a  on a Dell desktop  computer with Intel(R) Core(TM) i5-9500 CPU @ 3.00GHz  and 4.00 GB RAM.
	
	 The codes for the newGLMNET method and  the IRPN method  are collected from github\footnote{\url{https://github.com/ZiruiZhou/IRPN.}}. The code of the PNT  method can be obtained by modifying the code of the IRPN method. Note that the codes for the comparison algorithms are used to solve the $\ell_1$-regularization problem without bias term, while our goal is to solve the $\ell_1$-regularized logistic regression problem with the bias term $v$. For convenience, when applying the codes of these three comparison algorithms to solve the problem \eqref{1.1},  one may extend each instance with an additional dimension to eliminate this term \cite{C.L.L.2020,Y.T.2011}:
	\begin{equation*}
		\bar{{\bm{w}}}\gets \left[\begin{array}{c}
			{\bm{w}}\\
			v\\
		\end{array}\right] ,\ \bar{{\bm{a}}}_i\gets \left[
		\begin{array}{c}
			{\bm{a}}_i\\
			1\\
		\end{array}\right],\ i=1,\ldots,m, \ \bar{\bm{\lambda}}=\lambda \bar{\bm e},
	\end{equation*}
	where $\bar{\bm e}=[1,\ldots,1,0]^{\top}\in\mathbb{R}^{n+1}$. By the above transformations, the  problem \eqref{1.1}  is equivalent to the following one:
	\begin{equation}\label{4.2}
		\mathop {\min }\limits_{\bar{{\bm{w}}}\in \mathbb{R}^{n+1}}\ \left\{F(\bar{{\bm{w}}}):= \frac{1}{m}\sum_{i=1}^{m}\log(1+\exp(-{\bm{b}}_i\bar{{\bm{a}}}_i^{\top}\bar{{\bm{w}}}))+\bar{\bm{\lambda}}^{\top}|\bar{{\bm{w}}}|\right\}.
	\end{equation}
	For later discussion, we  define the loss function $g(\cdot)$ by $g(\bar{{\bm{w}}}):={1}/{m}\sum_{i=1}^{m}\log(1+\exp(-{\bm{b}}_i\bar{{\bm{a}}}_i^{\top}\bar{{\bm{w}}}))$. Moreover, we further denote the KKT residual function of  \eqref{4.2} by $r(\bar{{\bm{w}}}^k):=\|\bar{{\bm{w}}}-{\rm{Prox}}_{\bar{\bm{\lambda}}^{\top} |\cdot|}(\bar{{\bm{w}}}-\nabla g(\bar{{\bm{w}}}))\|$. Given a current iteration point $\bar{{\bm{w}}}^k\in\mathbb{R}^{n+1}$,  the first-order and second-order approximations of $F$ at $\bar{{\bm{w}}}^k$ are respectively denoted by
	\begin{equation*}
		\begin{split}
			&l_k(\bar{{\bm{w}}}):=g(\bar{{\bm{w}}}^k)+\nabla g(\bar{{\bm{w}}}^k)^{\top}(\bar{{\bm{w}}}-\bar{{\bm{w}}}^k)+\bar{\bm{\lambda}}^{\top} |\bar{{\bm{w}}}|,\\
			&q_k(\bar{{\bm{w}}}):=g(\bar{{\bm{w}}}^k)+\nabla g(\bar{{\bm{w}}}^k)^{\top}(\bar{{\bm{w}}}-\bar{{\bm{w}}}^k)+\frac{1}{2}(\bar{{\bm{w}}}-\bar{{\bm{w}}}^k)^{\top}{\bm{H}}_k(\bar{{\bm{w}}}-\bar{{\bm{w}}}^k)+\bar{\bm{\lambda}}^{\top} |\bar{{\bm{w}}}|,
		\end{split}
	\end{equation*}
	where ${\bm{H}}_k$ is  an approximation of the Hessian matrix of $g$. In addition, we define $g_k(\cdot)$  and $r_k(\cdot)$ by
	\begin{equation*}
		\begin{aligned}
			&g_k(\bar{{\bm{w}}}):=g(\bar{{\bm{w}}}^k)+\nabla g(\bar{{\bm{w}}}^k)^{\top}(\bar{{\bm{w}}}-\bar{{\bm{w}}}^k)+\frac{1}{2}(\bar{{\bm{w}}}-\bar{{\bm{w}}}^k)^{\top}{\bm{H}}_k(\bar{{\bm{w}}}-\bar{{\bm{w}}}^k),\\ 
			& r_k(\bar{{\bm{w}}}):=\|\bar{{\bm{w}}}-{\rm{Prox}}_{\bar{\bm{\lambda}}^{\top} |\cdot|}(\bar{{\bm{w}}}-\nabla g_k(\bar{{\bm{w}}}))\|.
		\end{aligned}
	\end{equation*}
	
	The frameworks of newGLMNET algorithm, IRPN algorithm and PNT algorithm can be respectively referred to \cite{Y.H.L.2012}, \cite{Y.Z.S.2019}, \cite{M.Y.Z.2022}. All three algorithms are second-order algorithms.

	Next, we focus on the AS strategy combined with the above second-order algorithm to generate the solution path of problem \eqref{4.2}. 
	For the one-variable problem \eqref{4.2}, the process of the AS strategy for generating the solution path can be referred to  \cite{L.Y.S.2020}.  The algorithm framework of the AS strategy  is shown in Algorithm \ref{al7}. 

	\begin{algorithm}[htb]
		\caption{AS strategy for solving problem \eqref{4.2}}	\label{al7}
		\begin{algorithmic}[1]
			\Require Given a sequence: ${\lambda}^1>\ldots>{\lambda}^t>0$ and  tolerance $\epsilon>0$.
			
			\Ensure A solution path: $\bar{{\bm{w}}}^*({\lambda}^1),\ldots,\bar{{\bm{w}}}^*({\lambda}^t).$
			
			\State {\bf{Initialization:}}  Generate an initial index set $ I^*(\lambda^0)$ by a screening rule.
			
			\For{$i=1,2,\ldots,t$} 
			
			\State Let $I^0(\lambda^i)=I^*(\lambda^{i-1}),$ find 
			\begin{equation}\label{7.2}
				\begin{split}
					\bar{{\bm{z}}}^0(\lambda^i)\approx\mathop {\arg\min }\limits_{\bar{{\bm{z}}}\in\mathbb{R}^{|{I}^0(\lambda^i)|}} \left\{g^{{I}^0(\lambda^i)}(\bar{{\bm{z}}})+({\bar{\bm{\lambda}}^i})_{{I}^0(\lambda^i)}^{\top}|\bar{{\bm{z}}}|  \right\}
				\end{split}	
			\end{equation}
			such that $\|r^{{I}^0(\lambda^i)}(\bar{{\bm{z}}}^0(\lambda^i))\|\leq \epsilon/\sqrt{2},$ where $\bar{\bm{\lambda}}^i=\lambda^i\bar{e}, \ \bar{{\bm{e}}}=\left[1,\ldots,1,0\right]^{\top}\in\mathbb{R}^{n+1}$. 
			Extend $\bar{{\bm{z}}}^{0}(\lambda^i)$ to $\bar{{\bm{w}}}^{0}(\lambda^i)$ as 
			$
			\bar{{\bm{w}}}^{0}(\lambda^i)_{I^0(\lambda^i)}=\bar{{\bm{z}}}^{0}(\lambda^i),\ \bar{{\bm{w}}}^{0}(\lambda^i)_{\bar{I}^0(\lambda^i)}={\bm{0}},
			$
			where $\bar{I}^0(\lambda^i)$ denotes the complement of ${I}^0(\lambda^i)$.
			
			\State Compute $r(\bar{{\bm{w}}}^0(\lambda^i))$ and set $l=0$.
			
			\While{$\|r(\bar{{\bm{w}}}^l(\lambda^i))\|>\epsilon$}
			\State Create $J^{l+1}(\lambda^i)$:
			$$
			J^{l+1}(\lambda^i)=\left\{j\in\bar{I}^l(\lambda^i)\ | \ -(\nabla g(\bar{{\bm{w}}}^l(\lambda^i)))_j\notin \left(\partial(\bar{\bm{\lambda}}^i)^{\top}|\bar{{\bm{w}}}^l(\lambda^i)|+ {\epsilon \mathcal{B}_{\infty}}/{\sqrt{2|\bar{I}^l(\lambda^i)|}}\right) _j \right\}.
			$$
			\State Update $I^{l+1}\gets I^{l}\cup J^{l+1}$.  Solve  the following problem: 
			\begin{equation}\label{7.3}
				\begin{split}
					\bar{{\bm{z}}}^{l+1}(\lambda^i)\approx\mathop {\arg\min }\limits_{\bar{{\bm{z}}}\in\mathbb{R}^{|{I}^{l+1}(\lambda^i)|}} \left\{g^{{I}^{l+1}(\lambda^i)}(\bar{{\bm{z}}})+({\bar{\bm{\lambda}}^i})_{{I}^{l+1}(\lambda^i)}^{\top}|\bar{{\bm{z}}}|  \right\}
				\end{split}
			\end{equation}
			such that $\|r^{{I}^{l+1}(\lambda^i)}(\bar{{\bm{z}}}^{l+1}(\lambda^i))\|\leq \epsilon/\sqrt{2},$ where $\bar{\bm{\lambda}}^i=\lambda^i\bar{{\bm{e}}}, \ \bar{{\bm{e}}}=\left[1,\ldots,1,0\right]^{\top}\in\mathbb{R}^{n+1}$. 
			Extend $\bar{{\bm{z}}}^{l+1}(\lambda^i)$ to $\bar{{\bm{w}}}^{l+1}(\lambda^i)$ as 
			$
			\bar{{\bm{w}}}^{l+1}(\lambda^i)_{I^{l+1}(\lambda^i)}=\bar{{\bm{z}}}^{l+1}(\lambda^i),\ \bar{{\bm{w}}}^{l+1}(\lambda^i)_{\bar{I}^{l+1}(\lambda^i)}={\bm{0}}.
			$
			
			\State Compute $r(\bar{{\bm{w}}}^{l+1}(\lambda^i))$ and set $l\gets {l+1}$.
			
			\EndWhile
			
			\State Set $\bar{{\bm{w}}}^*(\lambda^i)= \bar{{\bm{w}}}^l(\lambda^i)$.
			
			\EndFor
		\end{algorithmic}
	\end{algorithm}
	
	In  problems \eqref{7.2} and \eqref{7.3} of Algorithm \ref{al7}, the functions $ g^{{I}}(\cdot)$ and $ r^{{I}}(\cdot)$   are  respectively denoted by 
	\begin{equation*}
		\begin{split}
			&g^{{I}}(\bar{{\bm{z}}}):=({1}/{m})\sum_{i=1}^{m}\log(1+\exp(-{\bm{b}}_i({\bm{\bar{A}}}_{I}\bar{{\bm{z}}})_i),\ \forall  \bar{{\bm{z}}}\in\mathbb{R}^{|{I}|},\\
			&r^{{I}}(\bar{{\bm{z}}}):=\|\bar{{\bm{z}}}-{\rm{Prox}}_{({\bar{\bm{\lambda}}^i})_{{I}}^{\top}|\cdot|}(\bar{{\bm{z}}}-\nabla g^{{I}}(\bar{{\bm{z}}})\|, \ \forall \bar{{\bm{z}}}\in\mathbb{R}^{|{I}|},
		\end{split}
	\end{equation*}
	where $ {\bm{\bar{A}}}=\left[\bar{{\bm{a}}}_1,\ldots,\bar{{\bm{a}}}_m\right]^{\top}\in\mathbb{R}^{m\times (n+1)}$ and ${\bm{\bar{A}}}_{I}$ is the matrix consisting of the columns of ${\bm{\bar{A}}}$ indexed by $ {{I}}.$

	\subsection{Stopping Criteria and Parameter Settings}
	In this subsection, we specify the stopping criteria for the tested algorithms and set the parameters for each algorithm.
	
	\subsubsection{Stopping Criteria}
	In our numerical  experiments, the following relative KKT residual is used to measure the accuracy of approximate optimal solutions obtained by PPDNA and AS strategy:
	\begin{equation*}
		\begin{split}
			&R_{kkt1}:=\max\left\{\frac{\|{\bm{w}}-{\rm{Prox}}_{\lambda\|\cdot\|_1}({\bm{w}}-{\bm{A}}^{\top}{\bm{u}})\|}{1+\|{\bm{w}}\|+\|{\bm{A}}^{\top}{\bm{u}}\|},\frac{|{\bm{u}}^{\top}\boldsymbol{1}_m|}{1+|{\bm{u}}^{\top}\boldsymbol{1}_m|}\right\},\\ &{R_{kkt2}}:=\frac{\|{\bm{y}}-{\bm{Aw}}-v\boldsymbol{1}_m\|}{1+\|{\bm{y}}\|+\|{\bm{Aw}}+v\boldsymbol{1}_m\|},\\
			&R_{kkt}:=\max\left\{R_{kkt1},R_{kkt2}\right\}.
		\end{split}	
	\end{equation*}
	Combining \eqref{grad} and \eqref{aux}, we note that $\nabla h(y)-u=0$ always holds during the algorithm iterations, so we do not consider this stopping criterion.
	In the AS strategy, when PPDNA is used to solve the problem with smaller dimension, we use the following relative KKT to measure the accuracy of the optimal solution of the problem with smaller dimension:
	\begin{equation*}
		\begin{split}
			&R^I_{kkt1}:=\max\left\{\frac{\|{\bm{z}}-{\rm{Prox}}_{\lambda\|\cdot\|_1}({\bm{z}}-({\bm{A}}_{I})^{\top}{\bm{u}})\|}{1+\|{\bm{w}}\|+\|{\bm{A}}^{\top}{\bm{u}}\|},\frac{|{\bm{u}}^{\top}\boldsymbol{1}_m|}{1+|{\bm{u}}^{\top}\boldsymbol{1}_m|}\right\},\\ &{R^I_{kkt2}}:=\frac{\|{\bm{y}}-{\bm{A}}_{I}{\bm{z}}-v\boldsymbol{1}_m\|}{1+\|{\bm{y}}\|+\|{\bm{Aw}}+v\boldsymbol{1}_m\|},\\
			&R^I_{kkt}:=\max\left\{R^I_{kkt1},R^I_{kkt2}\right\}.
		\end{split}	
	\end{equation*}
	
	For a given accuracy tolerance ``tol", we  terminate the tested algorithms when $R_{kkt}\leq {\rm{tol}}$ or the number of iterations exceeds $500$. We start the {\sc PPDNA} and  {\sc Ssn} algorithm with an initial point $({\bm{w}}^0,v^0,{\bm{u}}^0)=({\bm{0}},0,-2\times10^{-7}{\bm{b}}/m)$ on real and random data sets, where the choice of ${\bm{u}}^0$ should satisfy ${\bm{u}}^0\in {\rm{dom}}h^*=\{{\bm{u}}\in\mathbb{R}^m| -(1/m){\bm{1}}_m< {\bm{u}}\circ {\bm{b}}<{\bm{0}} \}$. For the initial point of the {\sc PPDNA} in the AS strategy, our settings are the same as above.
	
	For the  comparison  algorithms of IRPN, PNT, and newGLMNET, we stop the algorithms when 
	\begin{equation*}
		R_{kkt}:=\frac{\|\bar{{\bm{w}}}-{\rm{Prox}}_{\bar{\bm{\lambda}}^{\top}|\cdot|}(\bar{{\bm{w}}}-\nabla g(\bar{{\bm{w}}}))\|}{1+\|\bar{{\bm{w}}}\|+\|\nabla g(\bar{{\bm{w}}})\|}\leq{\rm{tol}}
	\end{equation*}
	or the number of iterations exceeds $500$, and  the initial iteration point of the IRPN, PNT  and newGLMNET algorithm are  set to $\bm{0}$. 
	
	\subsubsection{Parameter Settings}
	For the PPDNA in Algorithm \ref{al1}, we initialize the parameter $\sigma_0=\gamma_0=40/\lambda$ for real data and random data. We below present the update rules for $\sigma_{k+1}$ and $\gamma_{k+1}$. The strategy for updating $\sigma_{k+1}$ and $\gamma_{k+1}$ are  $\sigma_{k+1}=\min\{5\times 10^4,\rho'\sigma_k\}$ and  $\gamma_{k+1}=\min\{5\times 10^4,\rho'\gamma_k\}$, respectively, where 
	\begin{equation*}
		\rho'=\left\{
		\begin{aligned}
			&1.01, &r^k<0.01,\\
			&1,	&r^k\geq 0.01,
		\end{aligned}\right.\quad 
		r^k=\left\{
		\begin{aligned}
			&1, & k=0,\\
			&(R_{kkt1})^k/(R_{kkt1})^{k-1}, & k\geq 1
		\end{aligned}\right.
	\end{equation*}
	and  $(R_{kkt1})^k$ represents the value of $R_{kkt1}$ at the $k$-th outer iteration. 
	
	In Algorithm \ref{al2}, we set $\mu=0.01$ and $\eta=0.6$. The stopping criterion of the CG algorithm at the $j$-th step  is chosen as $\|{\bm{\mathcal{H}}}_j({\bm{d}}^j)+\nabla\psi({\bm{u}}^{j})\|\leq \min\{0.005,\|\nabla\psi({\bm{u}}^{j})\|^{1.1}\}$. In addition, we terminate the {\sc Ssn} algorithm when the stopping criteria \eqref{eqn3} and \eqref{eqn4} are satisfied, where $\epsilon_k$ and $\delta_k$  are  chosen as $\epsilon_k=\delta_k=9/k^{1.01}$.
	
	In Algorithm \ref{al3}, the parameters of the PPDNA in the adaptive sieving strategy are the same as them by the above settings. The initial active index set $I^*(\lambda^0)$ is set below, which is borrowed from \cite{L.Y.S.2020}.  We first compute ${\bm{s}}_i:={|\left\langle  {\bm{a}}_i,{\bm{b}}\right\rangle |}/(\|{\bm{a}}_i\| \|{\bm{b}}\|)$. Then the initial active index set can be obtained by
	\begin{equation*}
		I^*(\lambda^0)=\left\{i\in\{1,...,n\}\ | \ {\bm{s}}_i \ \rm{is\ among \ the\ first \ [\sqrt{n}] \ largest \ of\ all} \right\}.
	\end{equation*}
	The specific screening rules  can be referred to the work of  \cite{F.L.2008}.
	
	For the newGLMNET algorithm, IRPN algorithm, and PNT algorithm, we set the same parameters as in references  \cite{Y.H.L.2012}, \cite{Y.Z.S.2019}, \cite{M.Y.Z.2022}, respectively.
	
	
	
	\subsection{Numerical Results for Random Data}\label{Numerical Results for Random Data}
	In this subsection, we compare our algorithms (PPDNA, AS strategy with PPDNA) with newGLMNET, IRPN and PNT on random data.  We set  $(m,n)=(200i,5000i),i=1,...,9$, where $m$ and $n$ denote the number of samples and features, respectively. We follow the way  \cite{KK.K.B.2007} to generate random data $({\bm{A}}, {\bm{b}})$. We first generate roughly equal numbers of positive and negative samples, each about half of the total number of samples. The features of positive samples are independent and identical distribution, drawn
	from a normal distribution  $\mathcal{N}(1,1)$, while the features of negative samples are also independent and identically distributed, drawn
	from a normal distribution  $\mathcal{N}(-1,1)$. Furthermore, the sparsity of matrix ${\bm{A}}$ is chosen to be $70\%$.
	
	Next, we discuss the choice of hyper-parameter  $\lambda$ in our numerical experiments. 
	Based on the first-order optimality condition of problem \eqref{1.1}, we compute a critical value $\hat{\lambda}_{\max}$, which is given by   
	$$
	\hat{\lambda}_{\max}=\left\| \frac{1}{m}{\bm{A}}^{\top}({\bf{1}}_m-p_{\log}(\log(m_{+}/m_{-}),{\bm{0}}))\right\| _{\infty},
	$$   
	where  $m_{+}$ and $m_{-}$ denote the number of positive and negative samples, respectively. By virtue of \cite{K.K.B.2007}, we know that if $\lambda\geq\hat{\lambda}_{\max}$, then the optimal solution to problem \eqref{1.1} can achieve the maximum sparsity, i.e., ${\bm{w}}={\bm{0}}$. Taking the above into consideration, we test three values of the hyper-parameter $\lambda$: $\lambda=0.5\hat{\lambda}_{\max},\ 0.1\hat{\lambda}_{\max},\ 0.05\hat{\lambda}_{\max}.$ 
	
	Table \ref{tab:4} presents the numerical results of PPDNA, AS strategy, newGLMNET, IRPN and PNT on random data.  The results shown in this table include the  iteration steps (iter), the relative KKT residuals ($R_{kkt}$), and the CPU time (time). For the hyper-parameter sequence $0.5\hat{\lambda}_{\max}>0.1\hat{\lambda}_{\max}>0.05\hat{\lambda}_{\max}$ in the Table \ref{tab:4}, we let the AS strategy with PPDNA generate a solution path of problem \eqref{1.1}.
	The time of the  AS strategy with PPDNA in Table \ref{tab:4} represents the running time of solving problem \eqref{1.1}. It is observed that all algorithms can successfully solve all instances with high accuracy. As shown in the Table \ref{tab:4}, the number of internal iterations of PPDNA is much less than other second-order algorithms. When $m,n$ are large, the running time of PPDNA is much faster than newGLMNET, IRPN and PNT. Specifically, for the Instance $9$, the time of the AS strategy with PPDNA hardly exceeds $10$ seconds, while  PNT needs more than $300$ seconds to obtain the approximate optimal solution, and the time for newGLMNET and IRPN to solve large scale problems even reaches $1000$ seconds. In most cases, PPDNA and AS strategy with PPDNA can obtain higher accuracy solutions than newGLMNET, IRPN, and PNT. 
	
	\begin{center}
		\setlength{\tabcolsep}{0.7pt}{
			\begin{longtable}{|c|c|ccc|ccc|ccc|}
				\captionsetup{width=0.9\textwidth}
				\caption{Numerical results of PPDNA, AS strategy with PPDNA,  newGLMNET, IRPN and PNT  on random data when ${\rm{tol}}=10^{-6}$.  ``a"= PPDNA, ``b" = AS strategy with PPDNA, ``c"=newGLMNET, ``d"=IRPN, ``e"=PNT. ``3(19)" means 3 outer iterations (the total number of inner iterations is 19), and times are shown in seconds}\label{tab:4}\\
				\hline	 \multirow{2}*{Case} &\multirow{2}*{Alg}& \multicolumn{3}{c|}{$\lambda=0.5\lambda_{\max}$} & \multicolumn{3}{c|}{$\lambda=0.1\lambda_{\max}$} & \multicolumn{3}{c|}{$\lambda=0.05\lambda_{\max}$}\\   
				& 	& \multicolumn{1}{c}{iter} &\multicolumn{1}{c}{time}& \multicolumn{1}{c|}{$R_{kkt}$} & \multicolumn{1}{c}{iter}& \multicolumn{1}{c}{time}&  \multicolumn{1}{c|}{$R_{kkt}$}& \multicolumn{1}{c}{iter}& \multicolumn{1}{c}{time}& \multicolumn{1}{c|}{$R_{kkt}$}\\ \hline
				\endfirsthead	
				\hline 
				\endfoot
				\hline	 \multirow{2}*{Case} &\multirow{2}*{Alg}& \multicolumn{3}{c|}{$\lambda=0.5\lambda_{\max}$} & \multicolumn{3}{c|}{$\lambda=0.1\lambda_{\max}$} & \multicolumn{3}{c|}{$\lambda=0.05\lambda_{\max}$}\\  
				& 	& \multicolumn{1}{c}{iter} &\multicolumn{1}{c}{time}& \multicolumn{1}{c|}{$R_{kkt}$} & \multicolumn{1}{c}{iter}& \multicolumn{1}{c}{time}&  \multicolumn{1}{c|}{$R_{kkt}$}& \multicolumn{1}{c}{iter}& \multicolumn{1}{c}{time}& \multicolumn{1}{c|}{$R_{kkt}$}\\ \hline
				\endhead
				& $a$	 &  3(19) 	 & 0.39 	 & 1.3e-07	 &  3(19) 	 & 0.30 	 & 4.5e-07	 &  3(21) 	 & 0.39 	 & 8.4e-07\\  
				& $b$	 & $3(19)$	 & 0.06 	 & 1.3e-07	 & $3(17)$	 & 0.05 	 & 3.8e-07	 & $4(23)$	 & 0.05 	 & 8.3e-08\\  
				1	 & $c$	 &  15(62) 	 & 0.54 	 & 7.1e-07	 &  22(230) 	 & 1.01 	 & 4.7e-07	 &  23(253) 	 & 1.09 	 & 5.5e-07\\  
				& $d$	 &  7(61) 	 & 0.35 	 & 3.1e-07	 &  8(234) 	 & 0.65 	 & 2.5e-07	 &  9(252) 	 & 0.71 	 & 3.4e-07\\  
				& $e$	 &  14(65) 	 & 0.61 	 & 3.8e-07	 &  16(212) 	 & 0.92 	 & 8.2e-07	 &  17(238) 	 & 1.00 	 & 6.7e-07\\  
				\hline	 & $a$	 &  3(19) 	 & 1.66 	 & 9.1e-08	 &  3(20) 	 & 1.79 	 & 2.8e-07	 &  3(19) 	 & 1.50 	 & 2.9e-07\\  
				& $b$	 & $3(19)$	 & 0.11 	 & 9.3e-08	 & $3(20)$	 & 0.23 	 & 2.8e-07	 & $3(17)$	 & 0.16 	 & 5.2e-07\\  
				2	 & $c$	 &  16(115) 	 & 4.09 	 & 6.8e-07	 &  23(462) 	 & 8.16 	 & 4.4e-07	 &  26(576) 	 & 10.27 	 & 5.4e-07\\  
				& $d$	 &  8(136) 	 & 2.80 	 & 8.3e-08	 &  10(459) 	 & 5.62 	 & 4.8e-07	 &  11(518) 	 & 6.27 	 & 6.4e-07\\  
				& $e$	 &  13(119) 	 & 3.84 	 & 5.3e-07	 &  17(453) 	 & 7.54 	 & 4.7e-07	 &  19(600) 	 & 8.83 	 & 3.4e-07\\  
				\hline	 & $a$	 &  3(19) 	 & 3.74 	 & 7.1e-08	 &  3(20) 	 & 3.77 	 & 1.4e-07	 &  3(18) 	 & 3.09 	 & 1.9e-07\\  
				& $b$	 & $3(19)$	 & 0.40 	 & 7.1e-08	 & $3(19)$	 & 0.64 	 & 2.0e-07	 & $3(18)$	 & 0.44 	 & 2.4e-07\\  
				3	 & $c$	 &  17(175) 	 & 14.26 	 & 5.4e-07	 &  30(747) 	 & 31.55 	 & 6.7e-07	 &  26(788) 	 & 29.68 	 & 7.7e-07\\  
				& $d$	 &  5(190) 	 & 7.80 	 & 8.5e-08	 &  12(729) 	 & 20.28 	 & 8.0e-07	 &  15(888) 	 & 24.78 	 & 4.6e-07\\  
				& $e$	 &  13(172) 	 & 13.19 	 & 4.0e-07	 &  17(766) 	 & 24.74 	 & 4.1e-07	 &  19(959) 	 & 30.40 	 & 3.4e-07\\  
				\hline	 & $a$	 &  3(18) 	 & 6.03 	 & 1.9e-07	 &  3(21) 	 & 7.20 	 & 1.7e-07	 &  3(20) 	 & 6.27 	 & 2.0e-07\\  
				& $b$	 & $3(19)$	 & 0.42 	 & 1.1e-07	 & $3(19)$	 & 0.64 	 & 2.0e-07	 & $3(20)$	 & 0.84 	 & 2.0e-07\\  
				4	 & $c$	 &  15(161) 	 & 27.35 	 & 4.3e-07	 &  27(914) 	 & 67.04 	 & 9.6e-07	 &  30(1051) 	 & 75.25 	 & 5.8e-07\\  
				& $d$	 &  9(189) 	 & 19.58 	 & 8.0e-08	 &  14(930) 	 & 49.51 	 & 5.7e-07	 &  15(895) 	 & 49.35 	 & 8.7e-07\\  
				& $e$	 &  12(146) 	 & 23.10 	 & 8.0e-07	 &  17(891) 	 & 54.16 	 & 7.4e-07	 &  19(1176) 	 & 65.92 	 & 3.3e-07\\  
				\hline	 & $a$	 &  3(19) 	 & 10.42 	 & 3.3e-07	 &  3(21) 	 & 12.18 	 & 1.1e-07	 &  3(20) 	 & 10.79 	 & 1.5e-07\\  
				& $b$	 & $3(20)$	 & 0.72 	 & 4.2e-08	 & $4(24)$	 & 2.08 	 & 2.6e-08	 & $3(19)$	 & 1.49 	 & 4.3e-07\\  
				5	 & $c$	 &  17(226) 	 & 59.59 	 & 4.2e-07	 &  30(1040) 	 & 134.36 	 & 7.0e-07	 &  41(1540) 	 & 186.74 	 & 7.1e-07\\  
				& $d$	 &  5(253) 	 & 28.90 	 & 8.8e-08	 &  17(1227) 	 & 109.42 	 & 8.7e-07	 &  22(1593) 	 & 143.10 	 & 4.2e-07\\  
				& $e$	 &  13(234) 	 & 51.38 	 & 3.1e-07	 &  18(1294) 	 & 118.60 	 & 3.3e-07	 &  19(1719) 	 & 142.80 	 & 3.2e-07\\  
				\hline	 & $a$	 &  3(18) 	 & 14.64 	 & 7.6e-08	 &  3(22) 	 & 20.32 	 & 2.4e-07	 &  4(20) 	 & 14.67 	 & 9.8e-09\\  
				& $b$	 & $3(18)$	 & 1.00 	 & 3.6e-08	 & $3(22)$	 & 3.27 	 & 3.2e-07	 & $3(18)$	 & 2.31 	 & 1.1e-07\\  
				6	 & $c$	 &  17(244) 	 & 102.16 	 & 4.0e-07	 &  39(1473) 	 & 295.16 	 & 6.7e-07	 &  40(1904) 	 & 331.55 	 & 7.7e-07\\  
				& $d$	 &  5(232) 	 & 45.30 	 & 3.2e-07	 &  20(1508) 	 & 209.32 	 & 8.3e-07	 &  26(2033) 	 & 276.15 	 & 8.8e-07\\  
				& $e$	 &  13(251) 	 & 87.15 	 & 3.3e-07	 &  18(1592) 	 & 210.88 	 & 3.9e-07	 &  18(1922) 	 & 235.62 	 & 6.8e-07\\  
				\hline	 & $a$	 &  3(20) 	 & 25.37 	 & 9.0e-08	 &  3(20) 	 & 25.69 	 & 9.9e-08	 &  3(18) 	 & 21.19 	 & 7.8e-08\\  
				& $b$	 & $3(20)$	 & 3.20 	 & 1.0e-07	 & $3(20)$	 & 4.41 	 & 8.1e-08	 & $3(17)$	 & 4.67 	 & 1.0e-07\\  
				7	 & $c$	 &  16(352) 	 & 165.57 	 & 4.1e-07	 &  40(1709) 	 & 487.71 	 & 5.2e-07	 &  37(1730) 	 & 465.89 	 & 8.2e-07\\  
				& $d$	 &  6(355) 	 & 89.69 	 & 2.8e-07	 &  24(1817) 	 & 388.72 	 & 7.1e-07	 &  25(1912) 	 & 397.09 	 & 8.1e-07\\  
				& $e$	 &  13(328) 	 & 143.27 	 & 6.3e-07	 &  18(1999) 	 & 359.60 	 & 3.3e-07	 &  18(2380) 	 & 398.68 	 & 5.7e-07\\  
				\hline	 & $a$	 &  3(20) 	 & 36.28 	 & 4.7e-08	 &  3(22) 	 & 41.92 	 & 5.4e-08	 &  3(18) 	 & 26.66 	 & 4.6e-08\\  
				& $b$	 & $3(20)$	 & 2.72 	 & 4.7e-08	 & $3(22)$	 & 7.25 	 & 5.0e-08	 & $3(17)$	 & 6.68 	 & 4.6e-08\\  
				8	 & $c$	 &  16(455) 	 & 265.63 	 & 3.9e-07	 &  46(2212) 	 & 857.55 	 & 9.4e-07	 &  37(1573) 	 & 662.80 	 & 9.7e-07\\  
				& $d$	 &  7(460) 	 & 157.39 	 & 3.4e-07	 &  29(2324) 	 & 673.85 	 & 8.6e-07	 &  24(1812) 	 & 542.71 	 & 7.7e-07\\  
				& $e$	 &  13(431) 	 & 227.83 	 & 5.0e-07	 &  18(2530) 	 & 582.16 	 & 2.6e-07	 &  19(2541) 	 & 595.26 	 & 9.0e-07\\  
				\hline	 & $a$	 &  3(19) 	 & 43.41 	 & 3.6e-08	 &  3(18) 	 & 43.15 	 & 6.3e-08	 &  3(18) 	 & 37.75 	 & 6.5e-08\\  
				& $b$	 & $3(19)$	 & 3.38 	 & 3.6e-08	 & $3(19)$	 & 7.38 	 & 6.3e-08	 & $3(18)$	 & 6.09 	 & 1.0e-07\\  
				9	 & $c$	 &  17(404) 	 & 386.61 	 & 3.5e-07	 &  47(2112) 	 & 1233.33 	 & 8.8e-07	 &  38(1633) 	 & 955.85 	 & 8.8e-07\\  
				& $d$	 &  7(459) 	 & 218.74 	 & 1.7e-07	 &  31(2516) 	 & 1008.03 	 & 7.1e-07	 &  23(1642) 	 & 699.41 	 & 7.1e-07\\  
				& $e$	 &  13(416) 	 & 319.88 	 & 3.1e-07	 &  17(2166) 	 & 707.63 	 & 9.1e-07	 &  19(2687) 	 & 836.93 	 & 9.2e-07
				
		\end{longtable}}
	\end{center}

		
		\begin{figure}[H]
			\centering
			\subfigure[]{
				\begin{minipage}[htbp]{0.45\linewidth}
					\includegraphics[height=50mm,width=60mm]{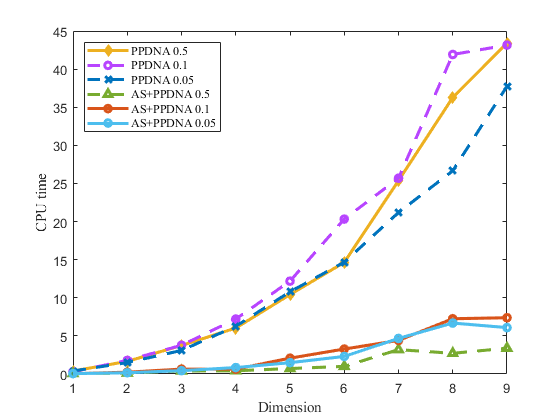}
			\end{minipage}}
			\subfigure[]{
				\begin{minipage}[htbp]{0.45\linewidth}
					\includegraphics[height=50mm,width=60mm]{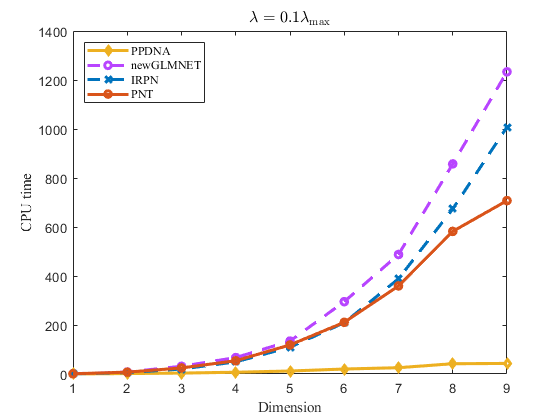}
			\end{minipage}}
			
			\caption{The time comparison of PPDNA, AS strategy with PPDNA, newGLMNET, IRPN, and PNT on  random data when ${\rm{tol}}=10^{-6}$}
			\label{fig:1}
		\end{figure}

		In order to more intuitively demonstrate the power of the AS strategy with the PPDNA, we show the time comparison between the PPDNA and the AS strategy with the PPDNA in the left panel of Fig.~\ref{fig:1}. We could find that the combination of the  PPDNA and the AS strategy can greatly improve the efficiency, and the running time is more than $5$ times faster than the oracle PPDNA. We demonstrate the efficiency of the PPDNA by showing the time comparison of PPDNA with other second-order solvers  in the right panel of Fig.~\ref{fig:1}. For the case with larger $m$ and $n$ in this figure, the running time of PPDNA is about $10$ times or more faster than that of newGLMNET, IRPN and PNT. Fig.~\ref{fig:1} reveals the excellent performance of the AS strategy in practice.

		\begin{sidewaystable}[p]
			\centering
			\caption{Numerical results of  AS strategy with PPDNA, AS  strategy with newGLMNET, AS strategy with IRPN and AS strategy with PNT on random data when ${\rm{tol}}=10^{-6}$.  ``a"= AS strategy with PPDNA, ``b" = AS strategy with newGLMNET, ``c"=AS strategy with IRPN, ``d"=AS strategy with PNT. Times are shown in seconds} \label{tab:6}
			\begin{tabular*}{\textheight}{@{\extracolsep\fill}cccccccc}
				\toprule%
				& & \multicolumn{1}{@{}c@{}}{time}& \multicolumn{1}{@{}c@{}}{nnx} &\multicolumn{1}{@{}c@{}}{$R_{kkt}$} & \multicolumn{1}{@{}c@{}}{iOuter}& \multicolumn{1}{@{}c@{}}{iInner}& \multicolumn{1}{@{}c@{}}{iAS} \\\cmidrule{3-3}\cmidrule{4-4}\cmidrule{5-5}\cmidrule{6-6}\cmidrule{7-7}\cmidrule{8-8}%
					i &$\lambda$&$a|b|c|d$&$a|b|c|d$&$a|b|c|d$&$a|b|c|d$&$a|b|c|d$&$a|b|c|d$\\ \midrule
				 & 0.5	 & 0.09$|$0.13$|$0.11$|$0.11  	 & 38$|$39$|$39$|$39  	 & 1.3e-07$|$1.9e-07$|$5.8e-08$|$1.0e-07  	 & 6$|$35$|$10$|$25  	 & 38$|$128$|$111$|$122  	 & 2$|$2$|$2$|$2  \\  
			1	 & 0.1	 & 0.14$|$0.38$|$0.28$|$0.33  	 & 96$|$97$|$97$|$97  	 & 3.8e-07$|$4.4e-07$|$2.5e-07$|$3.9e-07  	 & 9$|$63$|$24$|$48  	 & 48$|$559$|$468$|$545  	 & 3$|$3$|$3$|$3  \\  
			& 0.05	 & 0.20$|$0.58$|$0.41$|$0.47  	 & 109$|$110$|$110$|$110  	 & 8.3e-08$|$5.3e-07$|$2.7e-07$|$4.5e-07  	 & 11$|$67$|$25$|$50  	 & 62$|$699$|$650$|$699  	 & 2$|$2$|$2$|$2  \\  
			\hline	 & 0.5	 & 0.09$|$0.92$|$0.67$|$0.75  	 & 71$|$72$|$72$|$72  	 & 9.3e-08$|$1.4e-07$|$9.7e-08$|$8.9e-08  	 & 6$|$54$|$15$|$38  	 & 37$|$356$|$312$|$358  	 & 2$|$3$|$3$|$3  \\  
			2	 & 0.1	 & 0.42$|$2.20$|$1.56$|$1.78  	 & 175$|$176$|$176$|$176  	 & 2.8e-07$|$4.4e-07$|$2.5e-07$|$2.6e-07  	 & 9$|$59$|$28$|$53  	 & 59$|$1218$|$1145$|$1230  	 & 3$|$3$|$3$|$3  \\  
			& 0.05	 & 0.58$|$3.27$|$2.28$|$2.55  	 & 200$|$201$|$201$|$201  	 & 5.2e-07$|$5.4e-07$|$3.6e-07$|$3.3e-07  	 & 9$|$72$|$34$|$54  	 & 54$|$1554$|$1583$|$1617  	 & 2$|$2$|$2$|$2  \\  
			\hline	 & 0.5	 & 0.41$|$2.64$|$2.05$|$2.19  	 & 95$|$96$|$96$|$96  	 & 7.1e-08$|$1.5e-07$|$7.6e-08$|$1.3e-07  	 & 9$|$56$|$18$|$39  	 & 58$|$524$|$502$|$514  	 & 3$|$3$|$3$|$3  \\  
			3	 & 0.1	 & 0.97$|$7.22$|$4.97$|$5.39  	 & 231$|$232$|$232$|$232  	 & 2.0e-07$|$4.1e-07$|$2.3e-07$|$2.5e-07  	 & 9$|$83$|$35$|$56  	 & 58$|$1913$|$1929$|$1995  	 & 3$|$3$|$3$|$3  \\  
			& 0.05	 & 1.41$|$10.92$|$7.19$|$7.84  	 & 270$|$271$|$271$|$271  	 & 2.4e-07$|$5.1e-07$|$3.5e-07$|$4.2e-07  	 & 10$|$87$|$41$|$58  	 & 57$|$2581$|$2447$|$2670  	 & 2$|$2$|$2$|$2  \\  
			\hline	 & 0.5	 & 0.44$|$3.86$|$2.97$|$3.20  	 & 97$|$98$|$98$|$98  	 & 1.1e-07$|$1.2e-07$|$1.2e-07$|$7.7e-08  	 & 6$|$36$|$10$|$25  	 & 38$|$304$|$277$|$302  	 & 2$|$2$|$2$|$2  \\  
			4	 & 0.1	 & 1.08$|$14.19$|$9.73$|$10.14  	 & 284$|$285$|$285$|$285  	 & 2.0e-07$|$2.3e-07$|$2.5e-07$|$3.5e-07  	 & 6$|$86$|$42$|$55  	 & 39$|$2533$|$2643$|$2578  	 & 2$|$3$|$3$|$3  \\  
			& 0.05	 & 2.00$|$22.59$|$14.84$|$15.61  	 & 324$|$325$|$325$|$325  	 & 2.0e-07$|$3.3e-07$|$4.5e-07$|$3.4e-07  	 & 6$|$99$|$50$|$60  	 & 40$|$3260$|$3349$|$3381  	 & 2$|$2$|$2$|$2  \\  
			\hline	 & 0.5	 & 0.72$|$8.30$|$6.63$|$7.06  	 & 126$|$127$|$127$|$127  	 & 4.2e-08$|$1.1e-07$|$6.7e-08$|$1.0e-07  	 & 6$|$37$|$11$|$26  	 & 39$|$405$|$387$|$400  	 & 2$|$2$|$2$|$2  \\  
			5	 & 0.1	 & 2.80$|$28.28$|$21.14$|$22.01  	 & 337$|$338$|$338$|$338  	 & 2.6e-08$|$2.6e-07$|$3.1e-07$|$2.3e-07  	 & 10$|$90$|$50$|$58  	 & 63$|$3149$|$3448$|$3495  	 & 3$|$3$|$3$|$3  \\  
			& 0.05	 & 4.22$|$45.41$|$32.80$|$33.76  	 & 392$|$393$|$393$|$393  	 & 4.3e-07$|$5.5e-07$|$5.7e-07$|$4.9e-07  	 & 10$|$106$|$64$|$63  	 & 61$|$4531$|$4447$|$4699  	 & 2$|$2$|$2$|$2  \\  
			\hline	 & 0.5	 & 1.17$|$12.45$|$10.00$|$10.65  	 & 134$|$135$|$135$|$135  	 & 3.6e-08$|$1.0e-07$|$4.9e-08$|$9.4e-08  	 & 6$|$37$|$12$|$28  	 & 37$|$457$|$454$|$460  	 & 2$|$2$|$2$|$2  \\  
			6	 & 0.1	 & 4.50$|$50.75$|$34.34$|$34.88  	 & 387$|$388$|$388$|$388  	 & 3.2e-07$|$2.3e-07$|$3.1e-07$|$2.4e-07  	 & 9$|$118$|$60$|$60  	 & 63$|$3969$|$4243$|$4230  	 & 3$|$3$|$3$|$3  \\  
			& 0.05	 & 6.95$|$83.00$|$55.93$|$54.77  	 & 460$|$461$|$461$|$460  	 & 1.1e-07$|$4.5e-07$|$3.2e-07$|$3.1e-07  	 & 9$|$137$|$82$|$63  	 & 57$|$5633$|$6341$|$5930  	 & 2$|$2$|$2$|$2  \\  
			\hline	 & 0.5	 & 3.03$|$33.70$|$26.90$|$28.56  	 & 173$|$174$|$174$|$174  	 & 1.0e-07$|$9.9e-08$|$6.6e-08$|$5.9e-08  	 & 9$|$56$|$20$|$42  	 & 61$|$1099$|$1106$|$1137  	 & 3$|$3$|$3$|$3  \\  
			7	 & 0.1	 & 7.51$|$90.08$|$68.43$|$68.58  	 & 458$|$459$|$459$|$459  	 & 8.1e-08$|$2.3e-07$|$2.9e-07$|$2.0e-07  	 & 9$|$104$|$74$|$60  	 & 58$|$4923$|$5536$|$5594  	 & 3$|$3$|$3$|$3  \\  
			& 0.05	 & 12.28$|$158.77$|$99.94$|$114.54  	 & 531$|$532$|$533$|$532  	 & 1.0e-07$|$3.6e-07$|$6.7e-07$|$3.3e-07  	 & 9$|$130$|$85$|$61  	 & 52$|$6132$|$6677$|$7856  	 & 3$|$3$|$2$|$3  \\  
			\hline	 & 0.5	 & 2.50$|$35.21$|$25.68$|$27.15  	 & 199$|$200$|$200$|$200  	 & 4.7e-08$|$9.5e-08$|$8.6e-08$|$5.5e-08  	 & 6$|$37$|$14$|$28  	 & 40$|$766$|$767$|$794  	 & 2$|$2$|$2$|$2  \\  
			8	 & 0.1	 & 9.42$|$135.71$|$91.53$|$89.39  	 & 514$|$515$|$515$|$515  	 & 5.0e-08$|$2.1e-07$|$3.0e-07$|$2.6e-07  	 & 9$|$144$|$89$|$58  	 & 64$|$6411$|$7268$|$6892  	 & 3$|$3$|$3$|$3  \\  
			& 0.05	 & 16.20$|$231.65$|$155.16$|$158.75  	 & 588$|$589$|$589$|$590  	 & 4.6e-08$|$3.3e-07$|$3.7e-07$|$3.2e-07  	 & 9$|$121$|$77$|$62  	 & 50$|$5598$|$5660$|$9177  	 & 3$|$3$|$3$|$3  \\  
			\hline	 & 0.5	 & 4.52$|$59.79$|$36.47$|$38.12  	 & 193$|$194$|$194$|$194  	 & 3.6e-08$|$8.7e-08$|$6.0e-08$|$5.7e-08  	 & 6$|$39$|$14$|$27  	 & 37$|$744$|$732$|$751  	 & 2$|$2$|$2$|$2  \\  
			9	 & 0.1	 & 12.73$|$204.50$|$124.32$|$118.42  	 & 519$|$520$|$520$|$520  	 & 6.3e-08$|$2.7e-07$|$2.2e-07$|$3.0e-07  	 & 9$|$163$|$86$|$59  	 & 56$|$6837$|$7053$|$6949  	 & 3$|$3$|$3$|$3  \\  
			& 0.05	 & 18.82$|$292.59$|$183.49$|$180.66  	 & 608$|$609$|$608$|$608  	 & 1.0e-07$|$3.5e-07$|$6.4e-07$|$5.7e-07  	 & 9$|$142$|$85$|$61  	 & 54$|$6426$|$6560$|$9076  	 & 2$|$2$|$2$|$2  \\
				\botrule
			\end{tabular*}
		\end{sidewaystable}
		Since the adaptive sieving strategy is independent of the solver, for a fairer comparison,  we apply the AS strategy in combination with the algorithms  PPDNA, newGLMNET, IRPN, and PNT, respectively.	 Table \ref{tab:6} reveals the numerical results of the AS strategy with PPDNA, newGLMNET, IRPN, and PNT, respectively. Here,  AS strategy with PPDNA generates the solution path of problem \eqref{1.1}, while  AS strategy with newGLMNET,  AS strategy with IRPN and  AS strategy with PNT generate the solution path of the equivalent problem \eqref{4.2}. Likewise, we generate solution paths with $\lambda=0.5\hat{\lambda}_{\max},\ 0.1\hat{\lambda}_{\max},\ 0.05\hat{\lambda}_{\max}$. In this table, $``{\rm{nnx}}"$ denotes the number of non-zero components of the optimal solution ${\bm{w}}$ or $\bar{{\bm{w}}}$, and   ``time"  denotes the cumulative running time, $``{\rm{iOuter}}"$ represents the total number of outer iterations required for problem \eqref{3.2} and problem \eqref{3.3} to obtain an optimal solution for each problem $(P_{\lambda^i})$ in Algorithm \ref{al3}, $``{\rm{iInner}}"$ represents the total number of inner iterations required for problem \eqref{3.2} and problem \eqref{3.3} to obtain an optimal solution for each problem $(P_{\lambda^i})$ in Algorithm \ref{al3}, $``{\rm{iAS}}"$ denotes the number of sieving required to obtain an approximate solution to each problem $(P_{\lambda^i})$.

		\begin{figure}[H]
			\centering
			\includegraphics[width=80mm,height = 60mm]{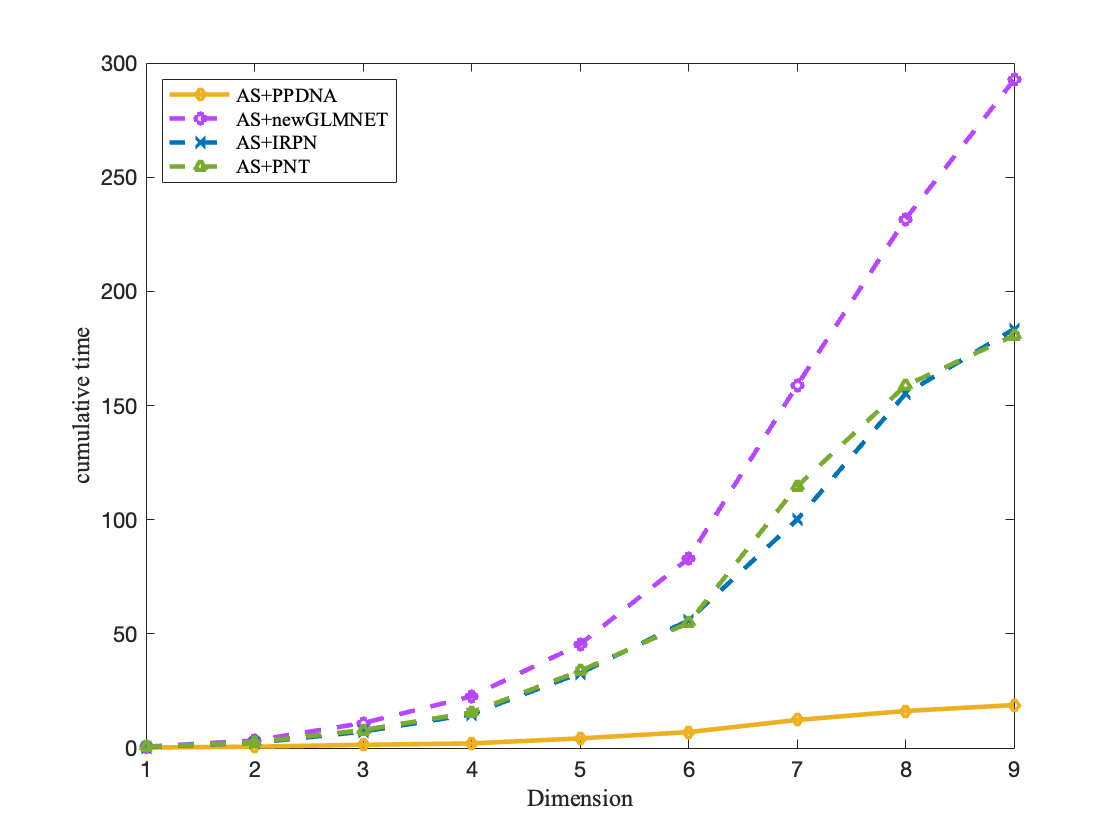}
			\caption{The total time comparison of   AS strategy with PPDNA,  AS  strategy with newGLMNET,  AS strategy with IRPN, and  AS strategy with PNT for generating solution paths of $\{0.5\hat{\lambda}_{\max},0.1\hat{\lambda}_{\max},0.05\hat{\lambda}_{\max}\}$ on  random data  when ${\rm{tol}}=10^{-6}$}
			\label{fig:2}
			\vspace{2mm}
		\end{figure}
		
		As revealed in Table \ref{tab:6} and Fig.~\ref{fig:2}, we could see that the AS strategy with the solver PPDNA significantly outperforms the strategies with the other solvers. The cumulative time of the AS strategy with PPDNA to generate  solution paths of $\{0.5\hat{\lambda}_{\max},0.1\hat{\lambda}_{\max},0.05\hat{\lambda}_{\max}\}$ is nearly $9$ times faster than that of the AS strategy with IRPN and the AS strategy with PNT, and $11$ times faster than that of the  AS strategy with newGLMNET, which means that even though the AS strategy speeds up the time of the solver to solve the problem, it cannot change the performance of the solver itself. Furthermore, we notice that the solution accuracy of the AS strategy with PPDNA is higher than other algorithms in Instance $8$ and Instance $9$. Regarding the sparsity of the solution, it could be found from this table that the number of nonzero components of the solution obtained by the AS strategy with PPDNA is one less than that of other algorithms. The reason is that the optimal solution $\bar{{\bm{w}}}$ of problem \eqref{4.2} contains the bias term $v$ and the optimal solution of $v$ is nonzero. In addition, we find that the total outer and inner iterations of  AS strategy with PPDNA are much less than those of AS strategy with other algorithms.
		
		\subsection{Numerical Results for Real Data}\label{Numerical Results for Real Data}
		\renewcommand\thefootnote{\arabic{footnote}}
		In this subsection, we compare the algorithms PPDNA, AS strategy with PPDNA, newGLMNET, IRPN and PNT for solving the $\ell_1$-regularized logistic regression problem \eqref{1.1} on real data sets $({\bm{A}},{\bm{b}})$. The six real data sets we tested are collected from LIBSVM data repository\footnote{\url{https://www.csie.ntu.edu.tw/~cjlin/libsvmtools/datasets/}}, UCI machine learning data repository\footnote{\url{https://archive.ics.uci.edu/ml/index.php}} and ELVIRA biomedical data repository\footnote{\url{http://leo.ugr.es/elvira/DBCRepository/index.html}}, including colon cancer,   leukemia, duke breast, arcene, gisette, and lungb. The statistics of all test data sets are shown in Table \ref{tab:1}, where ``nnz" and ``Des" denote the number of non-zero elements of matrix ${\bm{A}}$ and the density of  ${\bm{A}}$, respectively. We make the class label of each real data set $-1$ or $1$, and standardize the columns of the feature matrix ${\bm{A}}$ such that each column has a mean of $0$ and a variance of $1$. Similarly, we employ three values $\{0.5\hat{\lambda}_{\max},0.1\hat{\lambda}_{\max},0.05\hat{\lambda}_{\max}\}$ of the penalty parameter $\lambda$.  
		
		\begin{table}[htb]
			\centering
			\caption{Summary of tested data sets} \label{tab:1}
			\setlength{\tabcolsep}{6mm}{
				\begin{tabular}{ccccc}
					\hline	 \multirow{1}*{Data set}  &\multirow{1}*{Source}&\multirow{1}*{$m;n$}& \multicolumn{1}{c}{nnz} & \multicolumn{1}{c}{Des} \\
					\midrule 
					colonc	 & LIBSVM	 & $ 62;2000 $	 & 124000	 & 1.000\\  
					leukemia	 & LIBSVM	 & $ 38;7129 $	 & 270902	 & 1.000\\  
					duke breast	 & LIBSVM	 & $ 44;7129 $	 & 313676	 & 1.000\\  
					arcene	 & UCI	 & $ 100;10000 $	 & 991955	 & 0.992\\  
					lungb	 & ELVIRA	 & $ 181;12533 $	 & 2268470	 & 1.000\\  
					gisette	 & UCI	 & $ 1000;5000 $	 & 4714000	 & 0.943
					\\
					\hline
			\end{tabular}}
		\end{table}
		
		Table \ref{tab:2} reports the results of  the algorithms PPDNA, AS strategy with PPDNA, newGLMNET, IRPN, PNT  for solving problem \eqref{1.1} on real data when $R_{kkt} \leq 10^{-6}$. Similarly, in this table,  the  time of the AS strategy with PPDNA represents the CPU time to solve each problem  when the AS strategy generates a solution path of problem \eqref{1.1}. For the Case  \uppercase\expandafter{\romannumeral1}-\uppercase\expandafter{\romannumeral3}, we  notice that the running time of PPDNA is not much different from that of newGLMNET, IRPN, and PNT. But for the Case \uppercase\expandafter{\romannumeral4}-\uppercase\expandafter{\romannumeral6} with larger scale, the running time of PPDNA is significantly less than  other second-order algorithms. 
		Specifically, for Case \uppercase\expandafter{\romannumeral4} and \uppercase\expandafter{\romannumeral6}, PPDNA is about $3$ times faster than newGLMNET and PNT, about $2$ times faster than IRPN. From all the test results in Table \ref{tab:2}, we clearly see that combining the AS strategy  with the PPDNA can greatly improve the efficiency of solving problem \eqref{1.1}.

		\begin{center}
			\setlength{\tabcolsep}{1pt}{
				\begin{longtable}{|c|c|ccc|ccc|ccc|}
					\captionsetup{width=0.9\textwidth}
					\caption{Numerical results of PPDNA, AS strategy with PPDNA,  newGLMNET, IRPN and PNT  on real data when ${\rm{tol}}=10^{-6}$.  ``a"= PPDNA, ``b" = AS strategy with PPDNA, ``c"=newGLMNET, ``d"=IRPN, ``e"=PNT. ``4(18)" means 4 outer iterations (the total number of inner iterations is 18), and times are shown in seconds}\label{tab:2}\\
					\hline	 \multirow{2}*{Case} &\multirow{2}*{Alg}& \multicolumn{3}{c|}{$\lambda=0.5\lambda_{\max}$} & \multicolumn{3}{c|}{$\lambda=0.1\lambda_{\max}$} & \multicolumn{3}{c|}{$\lambda=0.05\lambda_{\max}$}\\ 
					& 	& \multicolumn{1}{c}{iter} &\multicolumn{1}{c}{time}& \multicolumn{1}{c|}{$R_{kkt}$} & \multicolumn{1}{c}{iter}& \multicolumn{1}{c}{time}&  \multicolumn{1}{c|}{$R_{kkt}$}& \multicolumn{1}{c}{iter}& \multicolumn{1}{c}{time}& \multicolumn{1}{c|}{$R_{kkt}$}\\ \hline
					\endfirsthead	
					\hline 
					\endfoot
					\hline	 \multirow{2}*{Case} &\multirow{2}*{Alg}& \multicolumn{3}{c|}{$\lambda=0.5\lambda_{\max}$} & \multicolumn{3}{c|}{$\lambda=0.1\lambda_{\max}$} & \multicolumn{3}{c|}{$\lambda=0.05\lambda_{\max}$}\\ 
					& 	& \multicolumn{1}{c}{iter} &\multicolumn{1}{c}{time}& \multicolumn{1}{c|}{$R_{kkt}$} & \multicolumn{1}{c}{iter}& \multicolumn{1}{c}{time}&  \multicolumn{1}{c|}{$R_{kkt}$}& \multicolumn{1}{c}{iter}& \multicolumn{1}{c}{time}& \multicolumn{1}{c|}{$R_{kkt}$}\\ \hline
					\endhead
					& $a$	 &  $4(18)$ 	 & 0.05 	 & 1.0e-07	 &  $5(27)$ 	 & 0.08 	 & 2.8e-07	 &  $4(25)$ 	 & 0.06 	 & 2.0e-07\\  
					& $b$	 &  4(18)	 & 0.02 	 & 1.0e-07	 &  5(27)	 & 0.02 	 & 2.8e-07	 &  4(24)	 & 0.02 	 & 2.2e-07\\  
					\uppercase\expandafter{\romannumeral1}	 & $c$	 &  17(42) 	 & 0.08 	 & 3.0e-07	 &  18(119) 	 & 0.12 	 & 4.2e-07	 &  19(166) 	 & 0.13 	 & 4.2e-07\\  
					& $d$	 &  4(28) 	 & 0.04 	 & 4.6e-07	 &  6(121) 	 & 0.08 	 & 2.2e-07	 &  7(173) 	 & 0.10 	 & 2.4e-07\\  
					& $e$	 &  11(36) 	 & 0.08 	 & 2.4e-07	 &  15(110) 	 & 0.10 	 & 2.6e-07	 &  14(149) 	 & 0.12 	 & 7.0e-07\\  
					\hline	 & $a$	 &  $4(16)$ 	 & 0.23 	 & 2.8e-07	 &  $4(19)$ 	 & 0.30 	 & 2.7e-07	 &  $4(20)$ 	 & 0.27 	 & 3.0e-07\\  
					& $b$	 &  4(16)	 & 0.00 	 & 2.8e-07	 &  4(19)	 & 0.01 	 & 2.7e-07	 &  4(18)	 & 0.02 	 & 3.1e-07\\  
					\uppercase\expandafter{\romannumeral2}	 & $c$	 &  15(47) 	 & 0.18 	 & 6.3e-07	 &  20(87) 	 & 0.29 	 & 3.1e-07	 &  21(130) 	 & 0.38 	 & 3.3e-07\\  
					& $d$	 &  5(41) 	 & 0.12 	 & 4.3e-08	 &  7(76) 	 & 0.19 	 & 7.3e-08	 &  7(115) 	 & 0.27 	 & 2.0e-07\\  
					& $e$	 &  13(46) 	 & 0.25 	 & 2.1e-07	 &  14(75) 	 & 0.30 	 & 6.8e-07	 &  14(120) 	 & 0.41 	 & 3.1e-07\\  
					\hline	 & $a$	 &  $3(18)$ 	 & 0.38 	 & 6.0e-07	 &  $5(25)$ 	 & 0.50 	 & 5.3e-07	 &  $5(25)$ 	 & 0.45 	 & 7.2e-07\\  
					& $b$	 &  3(18)	 & 0.00 	 & 6.0e-07	 &  5(25)	 & 0.02 	 & 5.3e-07	 &  6(26)	 & 0.01 	 & 2.6e-07\\  
					\uppercase\expandafter{\romannumeral3}	 & $c$	 &  17(47) 	 & 0.24 	 & 1.8e-07	 &  19(190) 	 & 0.55 	 & 4.5e-07	 &  20(174) 	 & 0.52 	 & 4.6e-07\\  
					& $d$	 &  5(34) 	 & 0.14 	 & 1.3e-07	 &  7(175) 	 & 0.43 	 & 2.7e-07	 &  7(192) 	 & 0.47 	 & 3.1e-07\\  
					& $e$	 &  13(39) 	 & 0.30 	 & 7.2e-07	 &  13(154) 	 & 0.51 	 & 9.0e-07	 &  14(179) 	 & 0.57 	 & 2.6e-07\\  
					\hline	 & $a$	 &  $10(31)$ 	 & 1.61 	 & 8.0e-07	 &  $16(50)$ 	 & 2.88 	 & 9.7e-07	 &  $7(32)$ 	 & 1.95 	 & 9.9e-07\\  
					& $b$	 &  11(33)	 & 0.05 	 & 6.3e-07	 &  19(55)	 & 0.22 	 & 6.2e-07	 &  10(39)	 & 0.13 	 & 6.3e-07\\  
					\uppercase\expandafter{\romannumeral4}	 & $c$	 &  18(527) 	 & 4.60 	 & 9.9e-07	 &  26(1068) 	 & 8.75 	 & 9.4e-07	 &  31(1360) 	 & 11.00 	 & 8.6e-07\\  
					& $d$	 &  8(577) 	 & 4.51 	 & 6.1e-07	 &  15(1087) 	 & 8.37 	 & 8.4e-07	 &  18(1357) 	 & 10.43 	 & 8.5e-07\\  
					& $e$	 &  10(546) 	 & 4.63 	 & 8.2e-07	 &  14(1136) 	 & 8.92 	 & 7.9e-07	 &  14(1666) 	 & 12.85 	 & 4.2e-07\\  
					\hline	 & $a$	 &  $4(18)$ 	 & 2.48 	 & 7.6e-07	 &  $4(21)$ 	 & 2.97 	 & 4.8e-07	 &  $4(20)$ 	 & 2.89 	 & 3.9e-07\\  
					& $b$	 &  5(20)	 & 0.05 	 & 7.9e-08	 &  4(21)	 & 0.09 	 & 4.8e-07	 &  4(20)	 & 0.08 	 & 3.9e-07\\  
					\uppercase\expandafter{\romannumeral5}	 & $c$	 &  15(83) 	 & 4.42 	 & 7.4e-07	 &  23(222) 	 & 8.29 	 & 2.9e-07	 &  23(184) 	 & 7.80 	 & 2.8e-07\\  
					& $d$	 &  5(85) 	 & 2.60 	 & 7.9e-08	 &  8(207) 	 & 5.20 	 & 1.8e-07	 &  9(166) 	 & 4.77 	 & 1.9e-07\\  
					& $e$	 &  13(86) 	 & 4.71 	 & 3.7e-07	 &  15(200) 	 & 7.07 	 & 4.6e-07	 &  15(160) 	 & 6.59 	 & 6.4e-07\\  
					\hline	 & $a$	 &  $4(19)$ 	 & 7.36 	 & 8.2e-08	 &  $3(19)$ 	 & 8.25 	 & 6.3e-07	 &  $3(25)$ 	 & 12.27 	 & 1.8e-07\\  
					& $b$	 &  4(19)	 & 0.44 	 & 8.2e-08	 &  3(19)	 & 2.69 	 & 6.3e-07	 &  3(25)	 & 4.84 	 & 1.8e-07\\  
					\uppercase\expandafter{\romannumeral6}	 & $c$	 &  17(37) 	 & 32.49 	 & 2.6e-07	 &  19(224) 	 & 43.57 	 & 5.5e-07	 &  20(307) 	 & 46.64 	 & 5.4e-07\\  
					& $d$	 &  4(25) 	 & 9.86 	 & 1.8e-07	 &  7(241) 	 & 24.12 	 & 3.5e-07	 &  8(322) 	 & 27.59 	 & 3.6e-07\\  
					& $e$	 &  11(30) 	 & 24.12 	 & 4.5e-07	 &  14(241) 	 & 39.19 	 & 3.6e-07	 &  15(314) 	 & 40.71 	 & 4.2e-07
			\end{longtable}}
		\end{center}

			\begin{sidewaystable}[p]
			\centering
			\caption{Numerical results of  AS strategy with PPDNA, AS strategy with newGLMNET, AS strategy with IRPN and AS  strategy with PNT on real data when ${\rm{tol}}=10^{-6}$. Case \uppercase\expandafter{\romannumeral1}-\uppercase\expandafter{\romannumeral6} represent datasets  colonc, leukemia,  duke breast, arcene, lungb and gisette, respectively. ``a"= AS strategy with PPDNA, ``b" = AS strategy with newGLMNET, ``c"=AS strategy with IRPN, ``d"=AS strategy with PNT. Times are shown in seconds} \label{tab:5}
			\begin{tabular*}{\textheight}{@{\extracolsep\fill}cccccccc}
				\toprule%
				& & \multicolumn{1}{@{}c@{}}{time}& \multicolumn{1}{@{}c@{}}{nnx} &\multicolumn{1}{@{}c@{}}{$R_{kkt}$} &\multicolumn{1}{@{}c@{}}{iOuter}&\multicolumn{1}{@{}c@{}}{iInner} &\multicolumn{1}{@{}c@{}}{iAS}   \\\cmidrule{3-3}\cmidrule{4-4}\cmidrule{5-5}\cmidrule{6-6}\cmidrule{7-7}\cmidrule{8-8}%
				Case &$\lambda$&$a|b|c|d$&$a|b|c|d$&$a|b|c|d$&$a|b|c|d$&$a|b|c|d$&$a|b|c|d$\\ \midrule
				 & 0.5	 & 0.05$|$0.06$|$0.05$|$0.05 	 & 5$|$6$|$6$|$6	 & 1.0e-07$|$2.9e-07$|$8.8e-08$|$2.6e-07	 & 4$|$34$|$9$|$22	 & 18$|$85$|$64$|$74	 & 1$|$2$|$2$|$2\\  
			\uppercase\expandafter{\romannumeral1}	 & 0.1	 & 0.09$|$0.13$|$0.09$|$0.11 	 & 25$|$26$|$26$|$26	 & 2.8e-07$|$4.4e-07$|$2.0e-07$|$2.6e-07	 & 9$|$37$|$12$|$27	 & 45$|$163$|$162$|$146	 & 2$|$2$|$2$|$2\\  
			& 0.05	 & 0.13$|$0.19$|$0.14$|$0.14 	 & 28$|$29$|$29$|$29	 & 2.2e-07$|$4.1e-07$|$2.5e-07$|$3.6e-07	 & 9$|$36$|$13$|$28	 & 51$|$280$|$284$|$260	 & 1$|$1$|$1$|$1\\  
			\hline	 & 0.5	 & 0.00$|$0.03$|$0.03$|$0.03 	 & 4$|$5$|$5$|$5	 & 2.8e-07$|$1.8e-07$|$3.2e-08$|$1.1e-07	 & 4$|$34$|$9$|$22	 & 16$|$92$|$73$|$80	 & 1$|$2$|$2$|$2\\  
			\uppercase\expandafter{\romannumeral2}	 & 0.1	 & 0.03$|$0.06$|$0.06$|$0.06 	 & 10$|$11$|$11$|$11	 & 2.7e-07$|$2.8e-07$|$1.4e-07$|$1.7e-07	 & 9$|$38$|$16$|$26	 & 39$|$214$|$230$|$202	 & 2$|$2$|$2$|$2\\  
			& 0.05	 & 0.06$|$0.11$|$0.09$|$0.11 	 & 13$|$14$|$14$|$14	 & 3.1e-07$|$3.1e-07$|$2.1e-07$|$2.3e-07	 & 8$|$45$|$14$|$27	 & 36$|$248$|$205$|$213	 & 2$|$2$|$2$|$2\\  
			\hline	 & 0.5	 & 0.00$|$0.03$|$0.03$|$0.03 	 & 10$|$11$|$11$|$11	 & 6.0e-07$|$2.6e-07$|$1.5e-07$|$2.3e-07	 & 3$|$34$|$10$|$24	 & 18$|$91$|$71$|$81	 & 1$|$2$|$2$|$2\\  
			\uppercase\expandafter{\romannumeral3}	 & 0.1	 & 0.03$|$0.09$|$0.09$|$0.08 	 & 25$|$26$|$26$|$26	 & 5.3e-07$|$4.6e-07$|$2.5e-07$|$2.6e-07	 & 9$|$37$|$13$|$27	 & 46$|$248$|$255$|$265	 & 2$|$2$|$2$|$2\\  
			& 0.05	 & 0.06$|$0.13$|$0.13$|$0.11 	 & 28$|$29$|$29$|$29	 & 2.6e-07$|$4.7e-07$|$2.9e-07$|$4.6e-07	 & 11$|$37$|$15$|$27	 & 51$|$330$|$347$|$342	 & 1$|$1$|$1$|$1\\  
			\hline	 & 0.5	 & 0.08$|$0.17$|$0.17$|$0.16 	 & 11$|$12$|$12$|$12	 & 6.3e-07$|$2.5e-07$|$2.0e-07$|$1.6e-07	 & 28$|$50$|$25$|$20	 & 77$|$2005$|$2021$|$2323	 & 2$|$2$|$2$|$2\\  
			\uppercase\expandafter{\romannumeral4}	 & 0.1	 & 0.31$|$0.95$|$0.75$|$0.78 	 & 55$|$56$|$56$|$56	 & 6.2e-07$|$5.3e-07$|$5.7e-07$|$4.1e-07	 & 25$|$48$|$23$|$26	 & 83$|$1639$|$1584$|$1703	 & 2$|$2$|$2$|$2\\  
			& 0.05	 & 0.44$|$1.70$|$1.37$|$1.44 	 & 64$|$65$|$65$|$65	 & 6.3e-07$|$5.5e-07$|$6.3e-07$|$4.0e-07	 & 29$|$57$|$34$|$28	 & 94$|$2700$|$2529$|$2917	 & 1$|$1$|$1$|$1\\  
			\hline	 & 0.5	 & 0.14$|$0.52$|$0.42$|$0.47 	 & 8$|$9$|$9$|$9	 & 7.9e-08$|$1.8e-07$|$8.1e-08$|$3.7e-07	 & 5$|$35$|$9$|$25	 & 20$|$155$|$133$|$143	 & 1$|$2$|$2$|$2\\  
			\uppercase\expandafter{\romannumeral5}	 & 0.1	 & 0.25$|$1.05$|$0.86$|$0.95 	 & 17$|$18$|$18$|$18	 & 4.8e-07$|$2.9e-07$|$1.5e-07$|$3.5e-07	 & 8$|$45$|$16$|$31	 & 41$|$384$|$389$|$363	 & 2$|$2$|$2$|$2\\  
			& 0.05	 & 0.36$|$1.58$|$1.33$|$1.42 	 & 19$|$20$|$20$|$20	 & 3.9e-07$|$2.9e-07$|$5.1e-07$|$3.1e-07	 & 8$|$48$|$18$|$34	 & 41$|$357$|$294$|$338	 & 2$|$2$|$2$|$2\\  
			\hline	 & 0.5	 & 0.50$|$7.84$|$5.42$|$6.00 	 & 3$|$4$|$4$|$4	 & 8.2e-08$|$2.6e-07$|$1.8e-07$|$4.5e-07	 & 5$|$34$|$8$|$22	 & 30$|$74$|$50$|$60	 & 2$|$3$|$3$|$3\\  
			\uppercase\expandafter{\romannumeral6}	 & 0.1	 & 3.36$|$27.96$|$15.38$|$19.35 	 & 54$|$55$|$55$|$55	 & 6.3e-07$|$5.5e-07$|$3.3e-07$|$3.6e-07	 & 10$|$55$|$20$|$45	 & 58$|$584$|$604$|$619	 & 3$|$3$|$3$|$3\\  
			& 0.05	 & 8.26$|$57.83$|$29.13$|$37.13 	 & 91$|$92$|$92$|$92	 & 1.8e-07$|$5.5e-07$|$3.3e-07$|$4.1e-07	 & 10$|$62$|$24$|$45	 & 72$|$927$|$983$|$948	 & 3$|$3$|$3$|$3
				\\
				\botrule
			\end{tabular*}
\end{sidewaystable}
				Similarly, in order to compare the performance of the four solvers combined with the AS strategy for generating the solution paths of problem \eqref{1.1}, we present the numerical results of the AS strategy in combination with the four solvers in Table \ref{tab:5}. In this table, $``{\rm{nnx}}"$ and ``time" represent the number of non-zero components of the optimal solution and the cumulative running time, respectively. We also show the results of $``{\rm{iOuter}}"$, $``{\rm{iInner}}"$, and $``{\rm{iAS}}"$, which have been defined in Section \ref{Numerical Results for Random Data}.
				Comparing the results in Table \ref{tab:2} and Table \ref{tab:5}, we observe that the solvers can effectively reduce the running time by using the adaptive sieving strategy. In the Case \uppercase\expandafter{\romannumeral5} with $\lambda=0.5\hat{\lambda}_{\max}$, the running time of the  AS  strategy with newGLMNET is about $8$ times faster than the oracle newGLMNET, and the running time of the  AS strategy with IRPN and  the AS strategy with PNT are nearly $6$ times and $10$ times faster than the oracle IRPN and the oracle PNT, respectively. We found that the $``{\rm{iOuter}}"$ and $``{\rm{iInner}}"$ of PPNDA+AS were less than those of newGLMNET+AS, IRPN+AS and PNT+AS. In addition, it is not hard to see in general, AS strategy only need sieving 2-3 times, which verifies the efficiency of the AS strategy.

		In summary, we  reasonably conclude that the PPDNA and the  AS strategy with the PPDNA are superior to other second-order algorithms for solving problem \eqref{1.1}.

		\section{Conclusions}\label{sec:5}
		In this paper, we develop an efficient  dual semismooth  Newton method based proximal point algorithm  for solving the $\ell_1$-regularized logistic regression problem. The  global and  asymptotically superlinear local convergence of the PPDNA  has been shown to hold under mild conditions. By fully exploiting the  sparse structure of the matrix, the computational cost of the {\sc Ssn} algorithm can be significantly reduced. Next, we combine the adaptive sieving strategy with the PPDNA to further improve the efficiency of solving a series of problems \eqref{1.1}. We develop this strategy to transform the problem into a smaller size problem, and then apply the PPDNA to solve it. Theoretical results verify that this  strategy can terminate in finite steps. Finally, the numerical results demonstrate the excellent performance of  the PPDNA and AS strategy compared to some state-of-the-art second-order algorithms.
		In the future work, we can focus on applying adaptive sieving strategy to convex optimization problems with fused lasso and cluster lasso or other sparse regularizers.

		\section*{Acknowledgement}
		The work of Yong-Jin Liu was in part supported by the National Natural Science Foundation of China (Grant No. 12271097) and the Key Program of National Science Foundation of Fujian Province of China (Grant No. 2023J02007).

		\section*{Declarations}
		\begin{itemize}
			\item {\small{\noindent \textbf{Conflict of interest} The authors declare that they have no conflict of interest.}}
			\item  {\small{\noindent \textbf{Data Availability} All data generated or analyzed during this study are included in this  article. The data that support the findings of this study are openly available in LIBSVM data repository (\url{https://www.csie.ntu.edu.tw/~cjlin/libsvmtools/datasets/}), UCI machine learning data repository (\url{https://archive.ics.uci.edu/ml/index.php}) and ELVIRA biomedical data repository (\url{http://leo.ugr.es/elvira/DBCRepository/index.html}).}} 
		\end{itemize}


\bibliography{sn-bibliography}

\end{document}